\documentclass[12pt,reqno]{amsart}

\usepackage{amssymb, amsmath, amsthm, amsfonts,amstext, graphicx}
\usepackage{mathtools} 
\mathtoolsset{showonlyrefs}
\usepackage{hyperref}
\numberwithin{equation}{section}

\usepackage[english]{babel}
\usepackage{bbold}
\usepackage{enumitem}
\usepackage{tikz,tkz-tab}
\usepackage{amsmath}
\usepackage{amsthm}
\usepackage{amssymb}
\usepackage{xcolor}
\usepackage{verbatim}
\usepackage{tikz}
\usetikzlibrary{calc}
\usepackage{mathtools}

\usepackage{bigints}
\newcommand\numberthis{\addtocounter{equation}{1}\tag{\theequation}}
\usepackage{stmaryrd}
\newcommand\Pa{\mathcal{P} }
\newcommand\La{\mathcal{L} }
\newcommand\R{\mathbb{R} }
\newcommand\N{\mathbb{N} }

\newcommand\Py{\mathbb{P} }

\usepackage{graphicx}
\newtheorem{theorem}{Theorem}
\newtheorem{lemme}{Lemma}
\newcommand{\vide}[1]{}
\newtheorem*{definition}{Definition}

\newtheorem{coroll}{Corollary}
\newtheorem{remark}{Remark}
\usepackage{comment}
\date{November 14, 2021}

\title[Averaging of diffusion processes on a simplex]{Averaging of semigroups associated to diffusion processes on a simplex}
\author{Dimitri Faure}
\address{DMA, École normale supérieure, Université PSL, CNRS, 75005 Paris, France}
\email{\tt{ dfaure@clipper.ens.psl.eu}}
\begin{document}
\maketitle

\begin{abstract}
We study the averaging of a diffusion process living in a simplex $K$ of $\mathbb R^n$, $n\ge 1$. We assume that its infinitesimal generator can be decomposed as a sum of two generators corresponding to two distinct timescales and that the one corresponding to the fastest timescale is pure noise with a diffusion coefficient vanishing exactly on the vertices of $K$. We show that this diffusion process averages to a pure jump Markov process living on the vertices of $K$ for the Meyer--Zheng topology. The role of the geometric assumptions done on $K$ is also discussed.\end{abstract}

\tableofcontents
\newpage
\section{Introduction}

Let $E$ be a set and consider a family of  $E$-valued Markovian processes $\left\{ X^{\gamma}\; , \; \gamma>0\right\}$ depending on some parameter $\gamma>0$ whose generators can be written in the form 
\begin{equation}
\label{eq:gengen}
\La_{\gamma}=\La^{(0)}+\gamma \La^{(1)}.
\end{equation}
A general question in the probabilistic and analysis literature is to understand the asymptotic behavior of $X^\gamma$ when $\gamma$ goes to infinity {\footnote{Alternatively, if $\varepsilon=\gamma^{-1}$, the aim is to understand in the long time scale $t \varepsilon^{-1}$ the behavior of the process generated by ${\mathcal L}^{(1)} + \varepsilon {\mathcal L}^{(0)}$, i.e. of the process which is a small perturbation of the process generated by ${\mathcal L}^{(1)}$.}} (see books and reviews \cite{SP08,Papa,BLP,Gar,Kif,Skoro,Olla,Ris,YZ,YZ2,FPS,EK,BG,KL,FW,Kam,Sch}). More precisely, given $x\in E$ and $t>0$,  we want to investigate the large $\gamma$ limit of the semigroup acting on a generic test function $f:E \to \mathbb R$ 
$$e^{t\La_{\gamma}}f(x)=\mathbb{E}_x \left(f(X_t^{\gamma})\right).$$

A general `meta-theorem' is that: 
\begin{equation}
\label{ResultatGeneral}
e^{t\La_{\gamma}}\underset{\gamma\to+\infty}{\longrightarrow}\Pa e^{t\La_{\infty}} \Pa 
\end{equation}
where $\Pa$ is the spectral projector onto the kernel of the dominant generator  $\La^{(1)}$, parallel to the image of $\La^{(1)}$, and $\La_{\infty}=\Pa \La^{(0)}\Pa$. When this quantity is well defined, one has: $$\Pa=\underset{t\to+\infty}{\lim} \frac{1}{t} \int_{0}^t e^{t\La^{(1)}}.$$ 

Of course such a theorem would require several assumptions to hold and may be invalid in several situations, in particular when the limit itself does not make sense. Observe also that a priori the operator ${\mathcal L}_\infty$ is not necessarily related to the generator of a Markovian process.\\

Let us give a formal proof of \eqref{ResultatGeneral}. Let denote $f_{t}^{\gamma}=e^{t\La_{\gamma}}f$. Dynkin's formula gives 
\begin{equation}
\label{a}
    \frac{\partial}{\partial t} f_{t}^{\gamma}= \left(\La^{(0)}+\gamma \La^{(1)}\right) f_{t}^{\gamma}.
\end{equation}
Looking for a solution of this problem of the form
\begin{equation}\label{b}
    f_{t}^{\gamma}=f_t^{(0)}+\frac{1}{\gamma} f_t^{(1)}+O\left(\frac{1}{\gamma^2}\right),
\end{equation}
injecting  $\eqref{b}$ in $\eqref{a}$ and equating coefficients in powers of $\frac{1}{\gamma}$ we get that for any $t>0$:
\begin{align}
& \La^{(1)} f_t^{(0)}=0, \numberthis \label{c}\\
    & \La^{(1)}f_t^{(1)}=\frac{\partial}{\partial t} f_t^{(0)}-\La^{(0)}f_t^{(0)} \numberthis \label{d}
\end{align}
Equation $\eqref{c}$ gives us that $f_t^{(0)}$ is in the kernel of $\La^{(1)}$, so eventually \begin{equation}\label{e}
    f_t^{(0)}=\Pa f_t^{(0)}
\end{equation} by the definition of $\Pa$. We have also $\Pa \La^{(1)}=0$ by the definition of $\Pa$, so applying $\Pa$ on the left on $\eqref{d}$, we get, using $\eqref{e}$ and $\Pa^{2}=\Pa$, that: \begin{align*}
    0=\Pa \La^{(1)} f_t^{(1)} 
    =\Pa\frac{\partial}{\partial t} f_t^{(0)}-\Pa\La^{(0)}f_t^{(0)} 
    =\frac{\partial}{\partial t} \Pa f_t^{(0)}-(\Pa\La^{(0)}\Pa) \Pa f_t^{(0)}.
\end{align*}
We thus have, using $f_0^{(0)}=f$, that: \begin{equation}\label{f}
    f_t^{(0)}=\Pa f_t^{(0)}=e^{t\La_{\infty}}\Pa f=\Pa e^{t\La_{\infty}} \Pa f.
\end{equation}
However, the proof is only formal, and it is only in some specific situations that a rigorous proof of the `meta-theorem' can be given. In the literature, this kind of theorem is sometimes called an `averaging principle' and has been developed not only in the context of diffusion processes \cite{SP08,Papa,Gar,Skoro,Olla,Ris,FW,Kam,Sch,k1,k2,k3,k4,k5,k6,k7,k8,k9,k10,k11,k12,k13,k14,k15,k16,k17,k18,k19,k20,k21,k22,k23,k24,k25} but also, and historically first, for dynamical systems \cite{Kam,a1,a2,a3,a4}.\\

For instance, the averaging problem has been widely studied for solutions $X^\gamma = (y^\gamma, z^\gamma)\in \mathbb{T}^{d-\ell}\times  \mathbb{T}^{\ell}$ of a stochastic differential equation (SDE) taking the form:
 \begin{equation}
 \label{g}
\left\{\begin{array}{ll}
dy_t^{\gamma}&=\gamma b^y\left(z_t^{\gamma},y_t^{\gamma}\right)\mathrm{d}t+\sqrt{\gamma}\sigma^y\left(z_t^{\gamma},y_t^{\gamma}\right)\mathrm{d}W_t, \\
     dz_t^{\gamma}&=b^z\left(z_t^{\gamma},y_t^{\gamma}\right)\mathrm{d}t+\sigma^z\left(z_t^{\gamma},y_t^{\gamma}\right)\mathrm{d}B_t   
     \end{array} \right. 
 \end{equation}
where all the functions $b^y, b^z, \sigma^y, \sigma^z$ involved are smooth and $(B_t)_{t\ge 0}$ and $(W_t)_{t\ge 0}$ are independent Brownian motions. Here, the generator of the coupled process $X^\gamma$ is in the form \eqref{eq:gengen} with the subdominant generator given by: $$\La^{(0)}=\langle b^z,\nabla_z \rangle+\frac{1}{2}\langle \sigma^z (\sigma^z)^T \nabla_z,\nabla_z \rangle$$ and the dominant generator driven by: $$\La^{(1)}=\langle b^y, \nabla_y \rangle+\frac{1}{2}\langle \sigma^y (\sigma^y)^T \nabla_y,\nabla_y \rangle.$$
Observe that in \eqref{g} there is a clear separation of time scales so that $z^\gamma$ (resp. $y^\gamma$) can be identified to a slow (resp. fast) component of the coupled process $X^\gamma$. Notice also that the first equation in $\eqref{g}$, when we fix $z_t^{\gamma}=\xi$, define an SDE on $\mathbb{T}^{d-\ell}$ of invariant distribution $\rho^{\infty}({\rm d} y,\xi)$.\\

When $\gamma$ is big enough, for any $t>0$, and for a small amount of time $\mathrm{d}t$, in a first approximation,  $z_{s}^{\gamma}$ is almost equal to $z_t^{\gamma}=\xi$ for $s\in [t, t+{\rm d}t]$. Hence, during this time interval, $(y_s^\gamma)_{s\in [t, t+{\rm d}t]}$ is roughly equal in law to the scaled process $({\tilde y}_{\gamma s})_{s\in [t,t+{\rm d} t]}$ where ${\tilde y}$ is the process on ${\mathbb T}^{d-\ell}$ with generator ${\mathcal L}^{(1)}$ in which the variable $z$ has been frozen to the value $\xi$ during this amount of time $\mathrm{d}t$. Therefore, for $\gamma$ large, by the ergodic theorem, the density of $y_s^{\gamma}$ is almost constant equal to $\rho^{\infty}(y,\xi)\mathrm{d}y$ during this time interval. It follows that in a second approximation $(z_t^{\gamma})_{t\ge 0}$ is equal in law to $(z_t)_{t\ge 0}$, the solution of $$\mathrm{d}z_t=F(z_t)\mathrm{d}t+A(z_t)\mathrm{d}B_t,$$
with  
$$\left\{\begin{array}{ll}
     F(\xi)&=\displaystyle\int_{\mathbb{T}^{d-\ell}} b^z(y, \xi)\rho^{\infty}({\rm d} y,\xi), \\
     A(\xi)A(\xi)^T&=\displaystyle \int_{\mathbb{T}^{d-\ell}} \sigma^z(y, \xi)(\sigma^z(y, \xi))^T \rho^{\infty}({\rm d} y,\xi).
\end{array} \right.$$
The generator of this autonomous process living in $\mathbb T^\ell$ is: $$\overline{\La}=\langle F, \nabla_z \rangle+\frac{1}{2}\langle A A^T \nabla_z,\nabla_z \rangle,$$ and therefore the law of $z_t^{\gamma}$ is close to the probability measure $\mu_t^z=\delta_z e^{t\overline{\La}}$ for large $\gamma$. This argument can be made rigorous and it is proved historically in \cite{k10,k11,k12,k13}, under some assumptions, that $(z_t^{\gamma})_{0\le t\le T}$ converges weakly to $(z_t)_{0\le t\le T}$ in $\mathcal{C}([0,T],\mathbb{T}^{\ell})$ when $\gamma$ tends to infinity for all $T>0$. \\

We now explain how to connect this result with the `meta-theorem' stated at the beginning of the paper. The result just above implies that 
for any smooth function $f: (y,z) \in \mathbb{T}^d={\mathbb T}^{d-\ell}\times {\mathbb T}^\ell  \to f(y,z) \in  \mathbb R$ we have that
\begin{equation*}
\lim_{\gamma \to \infty} \mathbb{E}_{(y,z)}(f(y_t^{\gamma},z_t^{\gamma})) = \int_{\mathbb T^\ell} \ \left[ \int_{{\mathbb T}^{d-\ell} } f(y',z') \rho^{\infty} ({\rm d} y',z')\right] \ {\rm d} \mu_t^z (z'). 
\end{equation*}
By the definition of $\Pa$ we have that: $$(\Pa f)(y,z)=\int_{\mathbb{T}^{d-\ell}} f(y',z)\rho^{\infty}({\rm d} y',z),$$
that is independent of $y$. Hence, denoting ${\bar f} (z) =({\mathcal P}f) (y,z)$, we get
\begin{equation*}
\lim_{\gamma \to \infty}  e^{t {\mathcal L}_\gamma} f  (y,z)= \lim_{\gamma \to \infty} \mathbb{E}_{(y,z)}(f(y_t^{\gamma},z_t^{\gamma})) = \int_{\mathbb T^\ell}  {\bar f} \ (z')\ {\rm d} \mu_t^z (z'). 
\end{equation*}
Recalling that $\mu_t^z =\delta_z\ e^{t{\overline{\mathcal L}}}$, we obtain by formal integration by parts:
\begin{equation*} \lim_{\gamma \to \infty} e^{t {\mathcal
L}_\gamma} f (y,z) = e^{t {\overline{\La}}} {\bar f} (z) \end{equation*}
On the other hand a trivial computation shows that:
\begin{align*}
\La_{\infty} f&=\Pa \La^{(0)}\Pa f =\Pa \La^{(0)} {\bar f}
=\overline{\La} {\bar f}.
\end{align*}
This proves the 'meta-theorem' \eqref{ResultatGeneral} for this particular system.\\

A fundamental remark is that in this seminal example, the kernel of the dominant process ${\mathcal L}^{(1)}$ is composed of the  functions constant in the second (slow) variable, i.e. depending only of the first (fast) variable $y$. The kernel of ${\mathcal L}^{(1)}$ is hence infinite-dimensional. A natural question motivating this paper is then: what happens for the 'meta-theorem' for diffusion processes if the kernel of ${\mathcal L}^{(1)}$ is finite dimensional generated by a basis $\{f_1, \ldots, f_{p}\}$? Our conjecture is that in general, the limiting semi-group $e^{t{\mathcal L}_\infty}$ appearing in the `meta-theorem', will be associated to a pure jump continuous time Markov process living in the space of the basis $\{f_1, \ldots, f_{p}\}$. As far as we know this question has not been addressed in the literature{\footnote{Of course the formulation of the question is a bit messy since we do not precise the functional spaces considered.}}. However, very recently, in \cite{BBCCNP}, the authors show that a certain class of diffusions, related to quantum continuous measurements \cite{q1,q2,q3,q4}, and having a generator in the form \eqref{eq:gengen} with a finite-dimensional kernel for ${\mathcal L}^{(1)}$, converge in the large $\gamma$ limit to a pure jump continuous time Markov process on a finite space{\footnote{In fact in \cite{BBCCNP} a more general situation is considered with a dominant, an intermediate and a subdominant generators.}}. To prove their theorem the authors develop a {\it{finite-dimensional}} homogenization theorem \cite[Theorem 3.1]{BBCCNP} for {\it{bounded operators}} and, due to the specific form of the diffusion processes considered (linear drift and quadratic mobility), are able to prove the convergence by developing a tricky perturbative argument. This homogenization theorem is very similar, at least in its simplified version, to the `meta-theorem' for matrices. A more natural approach, but technically much more involved, would have been in \cite{BBCCNP} to use directly their homogenization theorem for the infinitesimal generators of the diffusion processes, i.e. for unbounded linear operators (as it is done without mathematical rigor in \cite{BBT}). Unfortunately the proof of \cite[Theorem 3.1]{BBCCNP} seem to be difficult to extend for unbounded operators.\\

The aim of this work is to provide a first step in this direction by adapting the proof of  \cite[Theorem 3.1]{BBCCNP} for some (unbounded) generators of diffusion operators. In order that the kernel of ${\mathcal L}^{(1)}$ is finite-dimensional, specific properties of the corresponding diffusion process have to be imposed. We will consider diffusions living in a simplex $K\subset {\mathbb R}^n$. Moreover, we will only investigate the case where the dominant generator ${\mathcal L}^{(1)}$ does not contain any drift term, i.e. is pure noise, and such that the kernel of ${\mathcal L}^{(1)}$ is finite-dimensional. To satisfy the later condition we will assume that the volatility in ${\mathcal L}^{(1)}$ is non-negative and vanishes exactly on a finite subset $K_0$ of $K$. Extending the assumptions done here would be interesting but up to now we have not been able to do it. In this article we show that under a geometric assumption on $K_0$, the convergence of the diffusion processes on the simplex $K$ to a pure jump continuous-time Markov process on $K_0$ holds.  We also give a counterexample of the theorem if this geometrical hypothesis is not verified, and thus prove the optimality of our conditions.\\

The paper is structured as follows. In Section \ref{sec:DMR} we define and state precisely our main results and we prove them in Section \ref{sec:proofs}. Section \ref{sec:CE} is devoted to the counter-example showing in some sense the optimality of our geometric assumption on $K_0$. An appendix where a uniform ergodic theorem for martingales is proved concludes the paper.

%
%
%
%



\section*{Notations}
Let $n\ge 1$ be an integer. The set of real-valued $n\times n$-matrix is denoted  by $\mathcal{M}_n(\R)$ and the components of a matrix $A \in {\mathcal M}_n (\R)$ are denoted by $A_{i,j}$, $i, j \in \{1, \ldots, n\}$. In particular the coordinates of $x\in {\mathcal M}_{n,1}(\mathbb R)\approx \mathbb R^n$ are denoted by $x_i$, $i\in \{1, \ldots,n\}$. If $A=(A_{i,j})_{1\le i,j\le n} \in \mathcal{M}_n(\R)$, the trace of $A$ is denoted by  $\textbf{tr}(A)=\sum_{k=1}^n \limits A_{k,k}$. For a given matrix $A$ of $\mathcal{M}_n(\R)$, the complex conjugate of $A$ is written as $A^{\dagger}$. We also denote $a\cdot b=\sum_{i=1}^n a_i b_i$ the standard scalar product of $a$ and $b$ in $\R^n$ and  by $\|\cdot\|_2$ the associated norm.  For $x\in \R^n$ and $\varepsilon>0$, $B(x,\varepsilon)$ denotes the ball of center $x$ and radius $\varepsilon$ for the norm $\|\cdot\|_2$.

\bigskip
The supremum norm of a real valued function $f$ defined on a set $K\subset \mathbb R^n$ is denoted by $\|f\|_{\infty}=\underset{x\in K}{\sup} |f(x)|$. The set $\mathbb{D}(\R_+, K)$ is the space of c\`adl\`ag functions from $\R_+$ to $K\subset \mathbb R^n$. If $f\in \mathcal{C}^1(\R^n)$, we write $\nabla_x f=(\partial_i f)_{1\le i\le n}$ the gradient of $f$ and if furthermore $f\in \mathcal{C}^2(\R^n)$, we write $H_f=\left(\partial_{i,j}^2 f\right)_{1\le i,j\le n}$ the Hessian of $f$.  For a vector field $b: x \to \R^n \to b (x)=(b_1 (x), \ldots, b_n (x))  \in \R^n$, we write $\langle b, \nabla_x \rangle=\sum_{i=1}^n \limits b_i (x) \partial_i$. For a matrix $A$ of $\mathcal{M}_n(\R)$, one writes $\langle A \nabla_x,\nabla_x \rangle=\sum_{i,j=1}^n \limits A_{i,j} \partial_{i,j}^2$.

\bigskip
 For a probability distribution $\mu$ on $\R^n$ and a real-valued bounded function $f$ on $\R^n$, we write $\langle \mu, f \rangle=\int f \mathrm{d}\mu$. For a given differential operator $\La$ (in particular for the infinitesimal generator of a Markovian process), we denote by $t\mapsto e^{t\La}$ the associated semigroup. Hence, for a given probability measure $\mu$, an infinitesimal generator $\La$ of a Markov process and a time $t\ge 0$, we write $\mu e^{t\La}$ the law of the process of generator $\La$ at time $t$.\\

\section{Main result}
\label{sec:DMR}

  From now on $K$ is a simplex of $\R^n$.

\subsection{Definitions}

Let $\gamma>0$, we consider the diffusion processes \\$(X_t^{\gamma})_{t\ge 0}:=\left(X_t^{\gamma}(x)\right)_{t\ge 0}$ solution of the following SDEs on $K$ with initial condition $x\in K$: 
\begin{equation}
\label{EDS}
dX_t^{\gamma}=\sqrt{\gamma} \sigma(X_t^{\gamma})dW_t + b(X_t^{\gamma})dt+\sigma_0(X_t^{\gamma})dB_t, \quad X_0^{\gamma}(x)=x\in K 
\end{equation}
where $\sigma,\sigma_0:K\to \mathcal{M}_n(\R)$ and $b:K\to\R^n $ are Lipschitz functions on $K$ and $(W_t)_{t\ge 0}$ and $(B_t)_{t\ge 0}$ are two independent Wiener processes on $\R^n$. \\

\begin{remark}\label{coeffvar}
All the results proved in this article would still hold if we considered the diffusions processes solution of:$$dX_t^{\gamma}=\sqrt{\gamma} \sigma(X_t^{\gamma})dW_t + b^{\gamma}(X_t^{\gamma})dt+\sigma_0^{\gamma}(X_t^{\gamma})dB_t, \quad X_0^{\gamma}(x)=x\in K.$$
where: $$\left\{\begin{array}{ll}
\underset{\gamma\to+\infty}{\lim} \sigma_0^{\gamma}&=\sigma_0 \\
\underset{\gamma\to+\infty}{\lim} b^{\gamma}&=b.
\end{array}\right.$$
However, in order not to overcharge the proofs with purely technical difficulties, we will from now on only consider \eqref{EDS}. 
\end{remark}
We assume (see Remark \ref{r2}) that $\sigma,\sigma_0$ and $b$ are such that $X_t^{\gamma}(x)\in K$ for any $t\ge 0$ and any $\gamma\ge 0$ and thus this is also true for $(X_t)_{t\ge 0}:=(X_t(x))_{t\ge 0}$ the solution of the following SDE on $K$ with initial condition $x$: 
\begin{equation}
\label{eq:purenoise-007}
dX_t=\sigma(X_t)dW_t, \quad X_0(x)=x.
\end{equation}
We furthermore assume for the rest of the article that $\sigma$ is null only for a finite number of points and we denote $K_0=\{x\in K, \ \sigma(x)=0\}$.\\

The infinitesimal generator of the Markov process generated by $\eqref{EDS}$ is $\La_{\gamma}=\gamma \La^{(1)}+\La^{(0)}$ where \begin{equation}
\label{generateur}
     \La^{(0)}=\langle b, \nabla_x\rangle+\frac{1}{2}\langle \sigma_0\sigma_0^{\dagger} \nabla_x, \nabla_x\rangle \quad \text{and} \quad  \La^{(1)}=\frac{1}{2}\langle  \sigma\sigma^{\dagger} \nabla_x, \nabla_x\rangle
\end{equation}
are respectively the generators of the subdominant and dominant processes. 
\begin{remark}
\label{compacite}
Since $K$ is compact and $b,\sigma$ and $\sigma_0$ are Lipschitz, we have that these functions are bounded on $K$.  
\end{remark}
\begin{remark}\label{r1}
In view of the discussion in the introduction, equation $\eqref{EDS}$ is not completely generic, since we assumed that the dominant process is pure noise, i.e. does not have any drift term. This implies in particular that the dominant process is a martingale. 
\end{remark}
\begin{remark}\label{r2} The fact that for any $x\in K$ one has $X_t^{\gamma}(x)\in K$ and $X_t(x)\in K$ for arbitrary times implies several constraints on $\sigma,\sigma_0$ and $b$ on the boundary of $K$. More precisely the process stays in $K$ whatever the starting point is if and only if for any point $x$ of a side of the simplex, $\sigma(x)$ and $\sigma_0(x)$ are parallel to this side and the vector $b(x)$ is null or points to the interior of $K$. We have thus that $\sigma$ and $\sigma_0$ are null on the vertices of $K$ since they are there parallel to two different sides.
\end{remark}
\begin{remark}\label{convexe}
We may ask ourselves if the results proved in this paper would still hold if we only assumed that $K$ is a compact convex set of $\R^n$. However, problems arise quickly if there is a point $x$ of the boundary of $K$ not in $K_0$ where the curvature is strictly positive. Indeed, the dominant process starting from $x$ will then escape $K$ with a strictly positive probability whatever the value of $\sigma(x)$. Furthermore, even when we add a non-null drift pointing to the interior everywhere on the boundary, preventing our process from escaping $K$, we observe that $X^{\gamma}$ does not always converges in law to a process living on $K_0$ when $\gamma$ approaches infinity: the limit process can for instance live in the entire boundary of $K$. To force the limit process to live on $K_0$ when $K$ is not a polytope, we have to introduce a dominant drift of a particular form, which creates new conceptual and technical difficulties we will not deal with in this article.

\end{remark}

Before studying $(X_t^{\gamma})_{t\ge 0}$ we will first consider the dominant process $(X_t)_{t\ge 0}$. In Theorem \ref{ergodique} it is proved that for all $x$, the process $(X_t(x))_{t\ge 0}$ converges almost surely to a random variable $X_{\infty}:=X_{\infty}(x)\in K_0$ as $t$ goes to $+\infty$. We may thus consider for any $z\in K_0$ the function: \begin{equation}\label{H_z}H_z : x\in K \longmapsto \mathbb{P}(X_{\infty}(x)=z)\in [0,1].\end{equation} that gives the probability that $(X_t(x))_{t\ge 0}$ converges to $z$.\\

Now that all our objects are well defined we may state a version of the ergodic theorem for $(X_t)_{t\ge 0}$. The uniformity proved in this theorem is fundamental for the derivation of the main theorem.

\begin{theorem}\label{ergodique} (Uniform Ergodic Theorem)\\
Assume that for any $z\in K_0$, the function $H_z$ is continuous. Then, for any Lipschitz function  $f: K\to \mathbb R$, we have that: $$e^{t\La^{(1)}} f\underset{t\to+\infty}{\longrightarrow} \Pa f,$$
where the convergence is uniform in $x$ and the projector $\Pa$ is defined by: \begin{equation}\Pa f(x)=\sum_{z\in K_0} H_z(x) f(z). \label{formule ergodique}\end{equation}

\end{theorem}

\begin{proof}
The proof of this theorem is postponed to the Appendix. 
\end{proof}

\begin{remark}[A crucial example]
\label{example} 
The continuity hypothesis in Theorem \ref{ergodique} is for example fulfilled in a specific case that is the one we are interested in this article.\\

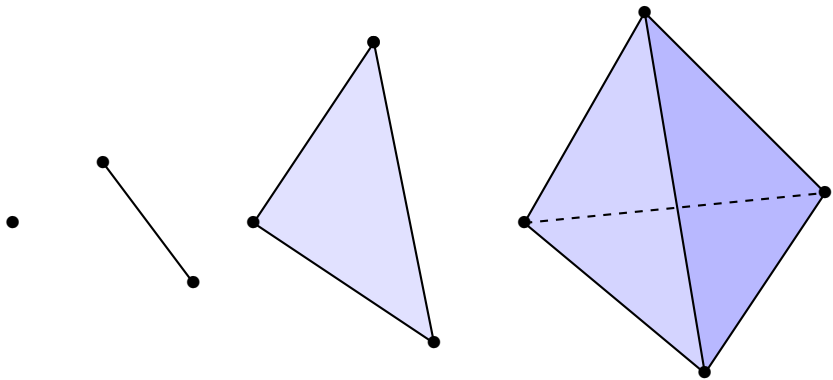
\begin{figure}[h]
\begin{tikzpicture}[scale=0.80]
\draw[black, ultra thick] (-0.5,0) node{$\bullet$} {};

\draw[black, ultra thick] (1,1) node{$\bullet$} {};
\draw[black, ultra thick] (2.5,-1) node{$\bullet$} {};
\draw[black,thick] (1,1) -- (2.5,-1);

\fill[color=blue!12] (3.5,0)--(5.5,3)--(6.5,-2)-- cycle;
\draw[black, ultra thick] (3.5,0) node{$\bullet$} {};
\draw[black, ultra thick] (5.5,3) node{$\bullet$} {};
\draw[black, ultra thick] (5.5,3) node{$\bullet$} {};
\draw[black, ultra thick] (6.5,-2) node{$\bullet$} {};
\draw[black,thick] (3.5,0) -- (5.5,3);
\draw[black,thick] (3.5,0) -- (6.5,-2);
\draw[black,thick] (5.5,3) -- (6.5,-2);

\fill[color=blue!17] (8,0)--(10,3.5)--(11,-2.5)-- cycle;
\fill[color=blue!28] (13,0.5)--(10,3.5)--(11,-2.5)-- cycle;
\draw[black, ultra thick] (8,0) node{$\bullet$} {};
\draw[black, ultra thick] (11,-2.5) node{$\bullet$} {};
\draw[black, ultra thick] (13,0.5) node{$\bullet$} {};
\draw[black, ultra thick] (10,3.5) node{$\bullet$} {};
\draw[black,thick] (8,0) -- (11,-2.5);
\draw[black,thick] (8,0) -- (10,3.5);
\draw[black,thick] (10,3.5) -- (11,-2.5);
\draw[black,thick] (13,0.5) -- (11,-2.5);
\draw[black,thick] (13,0.5) -- (10,3.5);
\draw[black,thick, dashed] (8,0) -- (13,0.5);
\end{tikzpicture}
\caption{A polytope of $\R^n$ has at least $n+1$ vertices.}
\end{figure}

Let us assume that the set $K_0$ is composed of $n+1$ points of $\R^n$ that are affinely independent, that is to say that $K_0$ is not included in an affine hyperplane of $\R^n$. Then the functions $H_z$ for $z\in K_0$ are affine functions. Indeed, since we have $n+1$ affinely independent points on a $n$-dimensional vector space, we get that for any $z\in K_0$, there exists an affine function $f_z$ such that for all $y\in K_0$: $$f_z(y)=\left\{\begin{array}{ll}
     0& \text{ if $y\neq z$,}  \\
     1& \text{ if $y=z$.}
\end{array} \right.$$
Since $(X_t(x))_{t\ge 0}$ is a martingale\footnote{More exactly, each coordinate of $(X_t(x))_{t\ge 0}$ is a real martingale.} living in a bounded space, it converges almost surely to a random variable $X_{\infty}(x)$ taking its values in $K_0$ (see the proof of the uniform ergodic theorem in the Appendix for more details). Thus for any $x\in K$: $$x=\mathbb{E}(X_0(x))=\mathbb{E}(X_{\infty}(x))=\sum_{y\in K_0} \mathbb{P}_x(X_{\infty}(x)=y)y.$$

Let $z\in K_0$, applying $f_z$ to last equation and using that $f_z$ is affine we get: 
\begin{equation*}
\begin{split}
    f_z(x)&=f_z\left(\sum_{y\in K_0} \mathbb{P}_x(X_{\infty}=y)y\right) 
    =\sum_{y\in K_0} \mathbb{P}_x(X_{\infty}=y) f_z(y)\\
    &=\mathbb{P}_x(X_{\infty}=z) =H_z(x).
    \end{split}
\end{equation*}
Furthermore, if we suppose that $K_0$ contains $n+1$ affinely independent points in a $n$-dimension space, we know actually that the points of $K_0$ are exactly the extremal points of the simplex $K$ since the noise is necessarily null on them (see Remark $\ref{r2}$), and a polytope of $\R^n$ has at least $n+1$ extremal points.
\end{remark}

\subsection{Statements}
Now that all the objects of the problem are now well defined, we may state our first main theorem.
\begin{theorem}
\label{mtheorem}
We assume that the cardinal of $K_0$ is equal to $n+1$ and the points of this set are affinely independent. Then for any Lipschitz function $f: K\to \mathbb R$ and any probability measure $\mu$ on $K$, we have for any $t>0$:
 $$\mathbb{E}_{\mu}\left(f\left(X_t^{\gamma}\right)\right) \underset{\gamma\to+\infty}{\longrightarrow} \mathbb{E}_{\overline{\mu}}\left(\overline{f}\left(\overline{X}_t\right) \right),$$
where $\overline{f}$ is the restriction of $f$ to $K_0$ and $(\overline{X}_t)_{t\ge 0}$ is the pure jump continuous time Markov process of the finite state space $K_0$ with initial distribution $\overline{\mu}$ and with generator $\overline{\La}$ given by 
$$\overline{\La}=\left(b(x)\cdot \nabla_x H_z(x) \right)_{x,z\in K_0}, \quad \overline{\mu}(z)=\int_{K} H_z(x)\mathrm{d}\mu(x) \text{ for $z\in K_0$}.$$
\end{theorem}

\begin{remark}
\label{reformulation}
Using the definition of $\Pa$ in $\eqref{formule ergodique}$, we have for any Lipschitz function $f:K\to \R$ that $\Pa f$ is the unique affine function on $K$ equal to $f$ on $K_0$. Thus, for any affine function $g$, we have that $\La_{\infty} g=\Pa \La^{(0)} \Pa g$ is affine and we can prove (see the proof of Theorem \ref{mtheorem}) that for $z\in K_0$: $$\La_{\infty} g(z)=\overline{\La} \overline{g}(z).$$ Eventually for any Lipschitz function $f$ and any $x\in K$: \begin{align*}
    \Pa e^{t\La_{\infty}} \Pa f(x)&=\sum_{z\in K_0} H_z(x) e^{t\La_{\infty}} \Pa f(z) =\sum_{z\in K_0} H_z(x) e^{t\overline{\La}} \overline{\Pa f} (z) \\
    &=\sum_{z\in K_0} H_z(x) e^{t\overline{\La}} \overline{f} (z),
\end{align*}
so Theorem \ref{mtheorem} is a reformulation of the `meta-theorem' $\eqref{ResultatGeneral}$.
\end{remark}
\begin{remark}\label{r4}
The geometrical hypothesis on $K_0$ may seem restrictive, but it is not clear what should be a more general statement. We will give in the last section a counter example of the theorem where $K\subset\R^2$ and $|K_0|=4$.
\end{remark}

Theorem \ref{mtheorem} provides the convergence of the semigroup $e^{t\La_{\gamma}}$ to the semigroup $e^{t\overline{\La}}$, that is to say the pointwise convergence in law of $(X_t^{\gamma})_{t\ge0}$ to $(\overline{X}_t)_{t\ge 0}$. It does not say anything about the convergence of $(X_t^{\gamma})_{t\ge 0}$ to $(\overline{X}_t)_{t\ge 0}$ at the path level. One has for every $\gamma>0$ that the paths of both processes $(X_t^{\gamma})_{t\ge 0}$ and $(\overline{X}_t)_{t\ge0}$ belong to $\mathbb{D}(\R_+, K)$. The natural topology on this space is the Skorokhod one, but we cannot in fact expect a weak convergence for this topology. Indeed as underlined in \cite[Theorem 13.4]{der}, weak convergence of processes with continuous paths in the Skorokhod topology yields a limiting process with continuous paths, and, while for $\gamma>0$, $(X_t^{\gamma})_{t\ge 0}$ has continuous paths almost surely, we have that $(\overline{X}_t)_{t\ge 0}$ does have discontinuous paths almost surely if $b$ is non-null on every vertex. To overcome this difficulty, we use the idea of \cite{BBCCNP} which is to replace the Skorokhod topology by the so-called Meyer--Zheng topology. \\

Let us define the Meyer--Zheng topology:
\begin{definition}
Consider a Euclidean space $(E,\|\cdot\|$) and denote by $\mathbb{L}^0=\mathbb{L}^0(\R_+,E)$ the space of $E$-valued Borel functions on $\R_+$. Given a sequence $\{w^\gamma \, ,  \ \gamma>0\}$ of elements of $\mathbb{L}_0(\R_+,E)$, the following assertions are equivalent and define the convergence in Meyer--Zheng topology of $\{ w^{\gamma}\ , \ \gamma > 0\}$ to $w\in \mathbb{L}^0(\R_+,E)$:
\begin{itemize}
    \item For all bounded continuous functions $f:\R_+\times E\to \R$: $$\underset{\gamma\to+\infty}{\lim}\int_0^{+\infty} f(t,w_t^{\gamma})e^{-t}\mathrm{d}t=\int_0^{+\infty} f(t,w_t)e^{-t}\mathrm{d}t.$$
    \item For $\lambda(\mathrm{d}t)=e^{-t}\mathrm{d}t$, we have for all $\varepsilon>0$:
    $$\underset{\gamma\to+\infty}{\lim} \lambda(\{s\in \R_+ | \|w_s^{\gamma}-w_s\|\ge \varepsilon\})=0.$$
    \item $\underset{\gamma\to+\infty}{\lim} d(w^{\gamma},w)=0$ where $d$ is defined by: $$d(w,w')=\int_0^{+\infty}\left\{1\wedge \|w_t-w_t'\|\right\}e^{-t}\mathrm{d}t.$$
\end{itemize}
The distance $d$ metrizes the Meyer--Zheng topology on $\mathbb{L}^0$ and $(\mathbb{L}^0,d)$ is a Polish space.
\end{definition}
We may now formulate our second main theorem:

\begin{theorem}\label{gtheorem}
Under the geometrical hypothesis of Theorem \ref{mtheorem}, we have that: $$\underset{\gamma\to+\infty}{\lim} X^{\gamma}=\overline{X}, \ \ \ \ \text{weakly in $\left(\mathbb{L}^0(\R_+,K),d\right)$}$$
where $(\overline{X}_t)_{t\ge 0}$ is the continuous-time pure jump Markov process on $K_0$ defined in Theorem \ref{mtheorem}. In other words, for all bounded continuous function $F:(\mathbb{L}^0(\R,K),d)\to (\R,|\cdot|)$, one has: $$\mathbb{E}\left(F\left(X^{\gamma}\right)\right)\underset{\gamma\to+\infty}{\longrightarrow}\mathbb{E}\left(F\left(\overline{X}\right)\right).$$
\end{theorem}

\section{Proof of Theorem 2 and Theorem 3}
\label{sec:proofs}

\subsection{Proof of Theorem 2}
Let us start with a quick sketch of the proof of Theorem \ref{mtheorem}.\\

We will first prove Lemma \ref{approx1}, a technical result that essentially says that for $t$ small and $\gamma$ big, our process of generator $\La_{\gamma}$ is basically the same as the one of generator $\La^{(1)}$. The keystone of our proof is however Corollary \ref{corollary}, that is proved using the uniform ergogic theorem combined with Lemma \ref{approx1}. Corollary \ref{corollary} says that for a given $t$ strictly positive, $X_t^{\gamma}$ de facto lives in small balls centered around the $z\in K_0$ when $\gamma$ is big enough. Therefore, when $\gamma$ is big enough, it is like $\left(X_t^{\gamma}\right)_{t\ge 0}$ is jumping from a ball to another. \\

Corollary 1 is true under the geometrical hypothesis of Theorem \ref{mtheorem}, since it is based on the ergodic theorem which is only proved true under it (thanks to Remark \ref{example}). However if this hypothesis is sufficient, it is not necessary: Corollary 1 would still hold if we only knew that the functions $H_z$ defined by $\eqref{H_z}$ are continuous. \\

Nevertheless, if we know that there is a jump process on $K_0$, we a priori do not know its form. Theorem \ref{mtheorem} says that, for a specific geometry, our process converges "weakly" to a continuous-time jump process with state space $K_0$ and generator $\overline{\La}$.\\

To prove that, we will first simplify our problem by approximating the drift $b$ by an affine drift $\widetilde{b}$ such that $b_{|K_0}=\widetilde{b}_{|K_0}$. This approximation is justified by using Corollary \ref{corollary}. For the stochastic process $\widetilde{X}_t^{\gamma}$ associated to the affine drift, the proof of our theorem for an affine test function $g$ but for a arbitrary initial condition $\mu$ is straightforward since we have thus that $\La_{\gamma} g=\La^{(0)}g$ is an affine function and $\Pa$ preserves affine functions. Hence we deduce the theorem for affine test functions even when the drift is not affine. \\

Then, using Corollary \ref{corollary} again, we know that $(X_t^{\gamma})_{t\ge 0}$ de facto lives in a neighborhood of $K_0$. Thus, instead of considering a generic test function $f$, we will consider an affine approximation $g$ of $f$ such that $f_{|K_0}=g_{|K_0}$. The function $g$ exists thanks to our geometrical hypothesis.\\

Eventually, since the theorem is true for $g$ and that one cannot distinguish $f(X_t^{\gamma})$ from $g(X_t^{\gamma})$ when $\gamma$ is big enough, it will still holds for $f$.\\

\bigskip

\begin{lemme}\label{approx1}
Let $f$ be a smooth function on $K$ and $\mu$ be a probability measure on $K$. There exist $C_1$ and $C_2$ two positive constants independent of $\mu$ such that for any $t,\gamma,h>0$, we have: $$\left|\langle \mu, e^{h\La_{\frac{\gamma}{h}}} f \rangle-\langle \mu, e^{\gamma \La^{(1)}}f \rangle \right|\le C_1 (h+h^2)^{\frac{1}{2}} e^{C_2\gamma t}.$$
\end{lemme}
\begin{proof}~\\
It is sufficient to prove that the inequality is fulfilled for $\mu=\delta_x$ for all $x\in K$. \\

Let $x$ be an arbitrary point of $K$ and recall that $X_t^{\frac{\gamma}{h}}:=X_t^{\gamma}(x)$ be the solution of $$\mathrm{d}X_t^{\frac{\gamma}{h}}=b(X_t^{\frac{\gamma}{h}})\mathrm{d}t+\sigma_0(X_t^{\frac{\gamma}{h}})\mathrm{d}W_t+\sqrt{\frac{\gamma}{h}}\sigma(X_t^{\frac{\gamma}{h}})\mathrm{d}B_t,$$
with initial condition $x$.\\

We now fix $\gamma$.\\

We couple the process $(Z_t^h)_{t\ge 0}=(Z_t^{h}(x))_{t\ge 0}$ solution of: $$Z_t^h=x+\frac{h}{\gamma}\int_0^t b(Z_u^h)\mathrm{d}u+\sqrt{\frac{h}{\gamma}}\int_0^t \sigma_0(Z_u^h)\mathrm{d}W_u+\int_0^{t} \sigma (Z_u^h) \mathrm{d}B_u,$$
with the process $(Z_t)_{t\ge 0}:=(Z_t(x))_{t\ge 0}$ solution of: $$Z_t=x+\int_0^{t} \sigma (Z_u)\mathrm{d}B_u,$$
through the common Brownian motion $(B_t)_t$ is the same in all these processes. The generator of $(Z_t^{h})_{t\ge 0}$ is $$\frac{h}{\gamma}\langle b(x),\nabla_x\rangle+\frac{1}{2}\frac{h}{\gamma}\langle \sigma_0\sigma_0^{\dagger} \nabla_x, \nabla_x\rangle+\frac{1}{2}\langle \sigma\sigma^{\dagger} \nabla_x, \nabla_x\rangle,$$ that is also the generator of $\left(X_{\frac{h}{\gamma}\times t}^{\frac{\gamma}{h}}\right)_{t\ge0}$ by scaling invariance of the Brownian motion. Since they share the same initial distribution, we have for all $h>0$ that:
$$\left(X_{\frac{h}{\gamma}\times t}^{\frac{\gamma}{h}}\right)_{t\ge 0}\overset{\La}{=}(Z_t^h)_{t\ge0},$$
which implies for $t=\gamma$ that: 
$$X_h^{\frac{\gamma}{h}}\overset{\La}{=}Z_{\gamma}^h.$$

Thus using Itô isometry, the fact that $b$ and $\sigma_0$ are bounded by a constant $M$ and that $\sigma$ is $k$-Lipschitz, we have that: \begin{align*}
&\mathbb{E}\left(\|Z_{\gamma}^h-Z_{\gamma}\|_2^2 \right)\\
&\le \mathbb{E}\left(\|\int_0^{\gamma} (\sigma(Z_u^h)-\sigma(Z_u))\mathrm{d}B_u +\sqrt{\tfrac{h}{\gamma}}\int_0^{\gamma} \sigma_0(Z_u^h)\mathrm{d}W_u  \right. \\ &\quad \quad \quad\quad \left. + \tfrac{h}{\gamma} \int_0^{\gamma} b(Z_u^h)\mathrm{d}u\|_2^2\right) \\
&\le 3\mathbb{E}\left(\|\int_0^{\gamma} (\sigma(Z_u^h)-\sigma(Z_u))\mathrm{d}B_u \|_2^2\right) + 3\mathbb{E}\left(\|\tfrac{h}{\gamma} \int_0^{\gamma} b(Z_u^h)\mathrm{d}u\|_2^2\right) \\
&\quad \quad    +3\mathbb{E}\left(\|\sqrt{\tfrac{h}{\gamma}}\int_0^{\gamma} \sigma_0(Z_u^h)\mathrm{d}W_u\|_2^2    \right)                       \\
&\le 3\mathbb{E}\left(\int_{0}^{\gamma} \|\sigma(Z_u^h)-\sigma(Z_u)\|_2^2  \mathrm{d}u\right)+3M^2\gamma^2 \tfrac{h^2}{\gamma^2} \\ & \quad \quad \quad +3\tfrac{h}{\gamma}\mathbb{E}\left(\int_0^{\gamma}\| \sigma_0(Z_u^h)\|_2^2\mathrm{d}u    \right)   \\ 
&\le 3k^2\mathbb{E}\left(\int_0^{\gamma} \|Z_u^h-Z_u\|_2^2\right)+3M^2(h+h^2).
\end{align*}
From this point, Grönwall's inequality gives us that for all $h>0$ we have: $$\mathbb{E}\left(\|Z_{\gamma}^h-Z_{\gamma}\|_2^2\right)\le 3M^2(h+h^2)e^{3k^2\gamma}.$$
Thus, for $f$ a smooth function, that is therefore $L$-Lipschitz on $K$, one has: \begin{align*}
    &\left|e^{h\La_{\frac{\gamma}{h}}}f(x)-e^{\gamma\La^{(1)}}f(x) \right|\\&=\left|\mathbb{E}(f(X_h^{\frac{\gamma}{h}}))-\mathbb{E}(f(Z_{\gamma})) \right|    \\&\le \mathbb{E}(\left|f(Z_{\gamma}^h)-f(Z_{\gamma}) \right|)\\
    &\le L\mathbb{E}\left(\|Z_{\gamma}^h-Z_{\gamma}\|_2 \right) \\
    &\le L\mathbb{E}\left(\|Z_{\gamma}^h-Z_{\gamma}\|_2^2 \right)^{\frac{1}{2}} \ \ \ \ \text{Using Cauchy-Schwarz} \\
    &\le \sqrt{3}LM(h+h^2)^{\frac{1}{2}}e^{\frac{3}{2}k^2\gamma}.
\end{align*}
Since the constants are independent of $x$ and $\gamma$, this proves the theorem for $C_1=\sqrt{3}LM$ and $C_2=\frac{3}{2}k^2$.
\end{proof}
Combining this result with the uniform ergodic theorem (Theorem \ref{ergodique}) gives us this crucial result:
\begin{coroll} \label{corollary}
For any $t>0$, $\eta>0$ and $\mu$ probability measure on $K$, we have that: $$\Py_{\mu}\left(X_t^{\gamma}\in \underset{z\in K_0}{\bigcup} B(z,\eta)\right)\underset{\gamma\to+\infty}{\longrightarrow} 1.$$
uniformly in $\mu$.
\end{coroll}
\begin{proof}~\\
We consider a smooth function $f:K\to[0,1]$ such that: \begin{itemize}
    \item $f=1$ on $\underset{z\in K_0}{\bigcup} B(z,\frac{\eta}{2})\cap K$.
    \item $f=0$ outside of $\underset{z\in K_0}{\bigcup} B(z,\eta)$.
\end{itemize}
Then we have: $$\langle \mu, e^{t\La_{\gamma}} f \rangle\le \Py_{\mu}\left(X_t^{\gamma}\in \underset{z\in K_0}{\bigcup} B(z,\eta)\right), $$ and thus it is sufficient to prove that the term on the left handside of the last display converges to $1$.\\

For a given positive $\gamma$ and a strictly positive $t$, we consider $\beta:=\beta(\gamma,t),h:=h(\gamma,t)$ such that: $$\gamma=C_1\beta^2e^{2C_2\beta t} \ \ \text{ and } \ \ h=\frac{1}{C_1\beta e^{2C_2\beta t}}.$$
Thus one has that $\gamma=\displaystyle\frac{\beta}{h}$, that $\beta\underset{\gamma\to+\infty}{\longrightarrow} +\infty$ and that $h\underset{\gamma\to+\infty}{\longrightarrow} 0$.\\

We recall $\eqref{H_z}$, i.e. that for all $z\in K_0$ and $x\in K$, we have $H_z(x)=\mathbb{P}(X_{\infty}(x)=z)$,  and therefore $\sum_{z\in K_0} \limits H_z=1$. We also recall $\eqref{formule ergodique}$, i.e. that for $x\in K$, we have $\Pa f(x)=\sum_{z\in K_0} \limits f(z)H_z(x)$. Eventually, since $f$ is constant equal to $1$ on $K_0$, we have that:
\begin{align*}
    \langle \mu, \Pa f \rangle&=\sum_{z\in K_0} f(z) \langle\mu, \Pa f \rangle \\
    &=\sum_{z\in K_0} \langle \mu, H_z \rangle \\
    &=\langle\mu,\sum_{z\in K_0} H_z \rangle \\
    &=\langle \mu, 1 \rangle \\
    &=1,
\end{align*}
for any $\mu$ probability measure on $K$. Thus one has $\langle \mu e^{(t-h)\La_{\frac{\beta}{h}}}, \Pa f \rangle=1$ for all $\gamma$ so using Markov property: \begin{align*}
|\langle \mu, e^{t\La_{\gamma}} f \rangle-1|=&\left|\langle \mu, e^{t\La_{\gamma}} f \rangle-\langle \mu e^{(t-h)\La_{\frac{\beta}{h}}}, \Pa f \rangle \right|\\&\le \left|\langle \mu e^{(t-h)\La_\frac{\beta}{h}}, e^{h\La_{\frac{\beta}{h}}} f \rangle-\langle \mu e^{(t-h)\La_\frac{\beta}{h}},e^{\beta\La^{(1)}} f \rangle\right|\\& \ \ \ \ \ \ +\left|\langle \mu e^{(t-h)\La_\frac{\beta}{h}}, e^{\beta\La^{(1)}}f\rangle-\langle \mu e^{(t-h)\La_\frac{\beta}{h}}, \Pa f\rangle\right| \\
\end{align*}
We now denote $\mu_{\gamma}=\mu e^{(t-h)\La_\frac{\beta}{h}}$ that is a probability measure. Then, using Lemma \ref{approx1}, we have that there exists $C_1$ and $C_2$ strictly positive constants independent of $\gamma$ and $\mu$ such that: \begin{align*}
    \left|\langle \mu_{\gamma}, e^{h\La_{\frac{\beta}{h}}} f \rangle-\langle \mu_{\gamma}, e^{\beta\La^{(1)}} f \rangle \right|&\le C_1 (h+h^2)^{\frac{1}{2}} e^{C_2 \beta t} \\
    &\le C_1\sqrt{h}e^{C_2\beta t}+C_1he^{C_2\beta t} \\
    &\le C_1 \sqrt{\frac{1}{C_1\beta e^{2C_2\beta t}}} e^{C_2\beta t}+\frac{1}{\beta}e^{-C_2\beta t}\\
    &\le \sqrt{\frac{C_1}{\beta}}+\frac{1}{\beta} \\
    &\underset{\gamma\to +\infty}{\longrightarrow} 0.
\end{align*}

On the other hand, since the convergence is uniform in $x$ in the uniform ergodic theorem, we have: $$\left|\langle \mu_{\gamma}, e^{\beta\La^{(1)}}f\rangle-\langle \mu_{\gamma}, \Pa f\rangle\right|\le\underset{x\in K}{\sup}\left|e^{\beta \La^{(1)}}f(x)-\Pa f(x)\right|\underset{\gamma\to+\infty}{\longrightarrow} 0,$$
uniformly in $\mu$. So eventually: $$\langle \mu,e^{t\La_{\gamma}} f \rangle\underset{\gamma\to+\infty}{\longrightarrow} 1,$$
and the convergence is uniform in $\mu$. This proves the corollary.
\end{proof}
Finally, to prove our theorem on affine functions for an affine drift, we need a last technical lemma: \begin{lemme}\label{elin}
We assume that $b:x\in \R^n\mapsto C x+d\in \R^n$ where $C\in \mathcal{M}_n(\R)$ and $d\in \R^n$. Then for any affine function $g:x\in \R^n\mapsto v\cdot x+l\in \R$ with $v\in \mathbb{R}^n$ and $l\in \mathbb{R}$, and any $t\ge 0$, one has that $e^{t\La_{\gamma}}g$ is affine. More precisely: \begin{equation}\label{expression}e^{t\La_{\gamma}}g:x\mapsto e^{tC^{\dagger}} v \cdot x+\left(v\cdot \int_0^{t} (e^{Cs}d) \mathrm{d}s +l\right)\end{equation}
\end{lemme}
\begin{proof} ~\\

We have that for all $x\in K$:
\begin{align*}
    e^{t\La_{\gamma}}g(x)&=\mathbb{E}_x\left(g(X_t^{\gamma})\right) \\
    &=\mathbb{E}_x\left(v\cdot X_t^{\gamma}+l \right) \\
    &=v \cdot \mathbb{E}_x(X_t^{\gamma})+l \numberthis \label{inter}.
\end{align*}
Applying $\mathbb{E}_x$ to the integral formulation of \eqref{EDS}, we get:
\begin{align*}
    \mathbb{E}_x(X_t^{\gamma})&=x+\int_0^t \mathbb{E}_x(b(X_t^{\gamma}))\mathrm{d}s \\
    &=x+\int_0^t \mathbb{E}(C X_t^{\gamma}+d)\mathrm{d}s\\
    &=x+dt+C\int_{0}^t \mathbb{E}_x(X_s^{\gamma})\mathrm{d}s.
\end{align*}
The associated differential equation is: \begin{equation}\label{edo}
    \frac{\partial}{\partial t} \mathbb{E}_x(X_t^{\gamma})=C\mathbb{E}_x(X_s^{\gamma})+d, \quad \mathbb{E}_x(X_0^{\gamma})=x.
\end{equation}
We can easily check that the unique solution to \eqref{edo} is: \begin{equation}\label{solution edo}
    \mathbb{E}_x(X_t^{\gamma})=e^{Ct}x+\int_0^{t} (e^{Cs}d) \mathrm{d}s
\end{equation}
Combining \eqref{inter} and \eqref{solution edo}, we eventually get:

\begin{align*}
    e^{t\La_{\gamma}}g(x)&=v\cdot\left(e^{tC} x \right)+v\cdot \int_0^{t} (e^{Cs}d) \mathrm{d}s +l 
\end{align*}
\end{proof}
\begin{remark}
\label{stabilite} 
 Let us consider an affine function $g$, equation \eqref{formule ergodique} gives us trivially $\Pa g=g$, that is to say $\Pa$ preserves affine functions.
\end{remark}

We may now prove our main theorem by following the path we mentioned earlier: 
\begin{proof} (Theorem \ref{mtheorem})\\
We have that $b$ is Lipschitz. Since the cardinal of $K_0$ is by hypothesis $n+1$ in a $n$-dimensional vector space, there exists a matrix $C=(c_{i,j})_{1\le i,j\le n}$ and a vector $d$ such that the affine maping: $$\widetilde{b}:x\in \R^n\mapsto Cx+d\in \R^n,$$
satisfies $\widetilde{b}(z)=b(z)$ for all $z\in K_0$. \\ 

We will now consider the system of equations: $$\left\{ \begin{array}{ll}
     X_t^{\gamma}&=Z_0+\displaystyle\int_{0}^t b(X_t^{\gamma})\mathrm{d}t+\int_0^t \sigma_0(X_t^{\gamma})\mathrm{d}B_t+\sqrt{\gamma}\int_0^t \sigma(X_t^{\gamma})\mathrm{d}W_t\\
     Y_t^{\gamma}&=Z_0+\displaystyle\int_{0}^t\widetilde{b}(Y_t^{\gamma})\mathrm{d}t+\int_0^t \sigma_0(Y_t^{\gamma})\mathrm{d}B_t+\sqrt{\gamma}\int_{0}^t\sigma(Y_t^{\gamma})\mathrm{d}W_t
\end{array}\right.$$
where the two independent Brownian motions $B_t$ and $W_t$ are the same for the two processes and $Z_0$ is a random variable of law $\mu$. By hypothesis, $X_t^{\gamma}$ stays in $K$ for arbitrary times, so Remark \ref{r2} gives us that $b(z)$ points to the interior of $K$ for all $z\in K_0$. But $K$ is a simplex, so each side of it is the convex envelope of its vertices, and the value of $\widetilde{b}$ on a point $x$ of this side is a convex combination of the values of $b$ on its vertices: $\widetilde{b}(x)$ therefore also points to the interior of $K$. Thus, Remark \ref{r2} gives us that $Y_t^{\gamma}$ stays in $K$ for arbitrary times. \\

We will denote the generators of these two processes by $\La_{\gamma}=\La^{(0)}+\gamma \La^{(1)}$ and $\mathcal{G}_{\gamma}=\mathcal{G}^{(0)}+\gamma \mathcal{G}^{(1)}$.\\

Let us now consider an affine function $g$. For any $\gamma>0$ and $t>0$ we have, using Lemma \ref{elin}, that $e^{t\mathcal{G}_{\gamma}}g$ is affine since $\widetilde{b}$ is affine. Hence $\mathcal{G}_{\gamma} e^{t\mathcal{G}_{\gamma}}g=\mathcal{G}^{(0)}e^{t\mathcal{G}_{\gamma}}g$ and: \begin{align*}
    \frac{\partial}{\partial t} e^{t\mathcal{G}_{\gamma}} g=\mathcal{G}_{\gamma} e^{t\mathcal{G}_{\gamma}} g  =\mathcal{G}^{(0)} e^{t\mathcal{G}_{\gamma}} g =\Pa \mathcal{G}^{(0)} \Pa e^{t\mathcal{G}_{\gamma}} g
\end{align*}
since $\Pa$ preserves linear functions (see Remark \ref{stabilite}).\\

Thus, writing $\mathcal{G}_{\infty}=\Pa\mathcal{G}^{(0)} \Pa$, we have that: $$e^{t\mathcal{G}_{\gamma}} g=e^{t\mathcal{G}_{\infty}} g=\Pa e^{t\mathcal{G}_{\infty}} g.$$

We denote $g_t=e^{t\mathcal{G}_{\infty}}g$ for $t\ge 0$. We have for $z\in K_0$ that: \begin{align*}\frac{\partial}{\partial t} \overline{g_t}(z)&=\frac{\partial}{\partial t} g_t(z)\\ 
&=\Pa \mathcal{G}^{(0)} \Pa g_t(z) \\
&=\sum_{y\in K_0} H_z(y) \mathcal{G}^{(0)} \Pa g_t(y) \quad \quad \text{Using \eqref{formule ergodique}} \\
&=\mathcal{G}^{(0)} \Pa g_t(z) \\
&=\widetilde{b}(z) \cdot \nabla_x \Pa g_t(z)\\
&=\widetilde{b}(z) \cdot \nabla_x \left(\sum_{y\in K_0} H_y(x) g_t(y) \right)(z) \\
&=\sum_{y\in K_0} \left( \widetilde{b}(z)\cdot \nabla_x H_y(z) \right) g_t(y) \\
&=\left(\widetilde{b}(x)\cdot \nabla_x H_y(x)\right)_{x,y\in K_0} (\overline{g_t}(y))_{y\in K_0} (z) \\
&=\overline{\mathcal{G}} \overline{g_t}(z).
\end{align*}
Eventually we have that $\overline{g_t}(z)=e^{t\overline{\mathcal{G}}}\overline{g}(z)$ for all $z\in K_0$, where we remind that $\overline{g}=g_{|K_0}$.\\

Finally we have for all probability measure $\mu$: \begin{align*}
    \langle \mu, e^{t\mathcal{G}_{\gamma}}g\rangle=\langle \mu \Pa, e^{t\mathcal{G}_{\infty}}g\rangle =\langle \overline{\mu}, e^{t\overline{\mathcal{G}}}\overline{g}\rangle
\end{align*}

However, if we look at the definition of $\overline{\La}$ and $\overline{\mathcal{G}}$ in Theorem \ref{mtheorem}, we notice that only the value of the drift on the $z\in K_0$ matters, therefore we have actually that $\overline{\La}=\overline{\mathcal{G}}$ so: $$\langle \overline{\mu}, e^{t\overline{\La}} \overline{g} \rangle=\langle \overline{\mu}, e^{t\overline{\mathcal{G}}} \overline{g} \rangle.$$

Thus, to prove that, for any affine function $g$, we have: $$\langle \mu, e^{t\La_{\gamma}} g \rangle\underset{\gamma\to+\infty}{\longrightarrow}\langle \overline{\mu}, e^{t\overline{\La}} \overline{g} \rangle,$$
we have to show that: $$\left|\langle \mu, e^{t\La_{\gamma}} g \rangle- \langle \mu, e^{t\mathcal{G}_{\gamma}} g \rangle\right|=\left|\langle \mu, e^{t\La_{\gamma}} g \rangle-\langle \overline{\mu}, e^{t\widetilde{\mathcal{G}}} \overline{g} \rangle\right|\underset{\gamma\to+\infty}{\longrightarrow} 0.$$ 

First of all, let us use Dynkin's formula for $g:x\mapsto e_i\cdot x$. We have thus $\nabla_x g=e_i$ and $H_g=0$. Therefore, writing $X_t^{\gamma}=\left([X_{t}^{\gamma}]_i\right)_{1\le i\le n}$ and $Y_t^{\gamma}=\left([Y_{t}^{\gamma}]_j\right)_{1\le j\le n}$ we have:\begin{align*}
    &\mathbb{E}_{\mu}\left([X_{t}^{\gamma}]_i-[Y_{t}^{\gamma}]_i \right)\\&=\mathbb{E}_{\mu}\left(g(X_t^{\gamma})\right)-\mathbb{E}_{\mu}\left(g(Y_t^{\gamma})\right)                 \\
    &=\langle\mu,g \rangle -\langle\mu,g\rangle \\& \quad   +\mathbb{E}\left( \int_{0}^t \left(b(X_s^{\gamma})\cdot\nabla_x g(X_s^{\gamma}) -\widetilde{b}(Y_s^{\gamma})\cdot\nabla_x g(Y_s^{\gamma})\right)\mathrm{d}s\right) \\& \quad \quad+  \frac{1}{2}\mathbb{E}\left(\int_0^{t}(\textbf{tr}\left(\sigma(X_s^{\gamma})\sigma(X_s^{\gamma})^{\dagger}H_g(X_s^{\gamma})^{\dagger}  \right)\right. \\ &\quad \quad \quad \quad \quad  \quad \quad \left.-\textbf{tr}\left(\sigma(Y_s^{\gamma})\sigma(Y_s^{\gamma})^{\dagger})H_g(Y_s^{\gamma})^{\dagger}  \right) \mathrm{d}s\right) \\
     & \quad \quad \quad
     +\frac{1}{2}\mathbb{E}\left(\int_0^{t}(\textbf{tr}\left(\sigma_0(X_s^{\gamma})\sigma_0(X_s^{\gamma})^{\dagger}H_g(X_s^{\gamma})^{\dagger}  \right) \right. \\ & \quad \quad \quad \quad \quad  \quad \quad \quad\left. -\textbf{tr}\left(\sigma_0(Y_s^{\gamma})\sigma_0(Y_s^{\gamma})^{\dagger})H_g(Y_s^{\gamma})^{\dagger}  \right) \mathrm{d}s\right) \\
    &=\mathbb{E}_{\mu}\left(\int_{0}^t (b(X_s^{\gamma})-\widetilde{b}(Y_s^{\gamma})) \cdot e_i  \mathrm{d}s\right) \\
    &=\mathbb{E}_{\mu}\left(\int_{0}^t (\widetilde{b}(X_s^{\gamma})-\widetilde{b}(Y_s^{\gamma})) \cdot e_i  \mathrm{d}s\right)+\mathbb{E}_{\mu}\left(\int_{0}^t (b(X_s^{\gamma})-\widetilde{b}(X_s^{\gamma})) \cdot e_i  \mathrm{d}s\right) \\
    &=\mathbb{E}_{\mu}\left(\int_{0}^t \sum_{j=1}^n c_{i,j}([X_{t}^{\gamma}]_j-[Y_{t}^{\gamma}]_j)  \mathrm{d}s\right) \\ &\quad \quad \quad+\mathbb{E}_{\mu}\left(\int_{0}^t (b(X_s^{\gamma})-\widetilde{b}(X_s^{\gamma})) \cdot e_i  \mathrm{d}s\right) \\
    &=\sum_{j=1}^n c_{i,j}\int_{0}^t \mathbb{E}_{\mu}\left([X_{t}^{\gamma}]_j-[Y_{t}^{\gamma}]_j \right)\mathrm{d}s +\int_{0}^t \mathbb{E}_{\mu}\left((b(X_s^{\gamma})-\widetilde{b}(X_s^{\gamma})) \cdot e_i  \right)\mathrm{d}s.
\end{align*}
Thus we have:
\begin{align*}
    &\left|\mathbb{E}_{\mu}\left([X_{t}^{\gamma}]_i-[Y_{t}^{\gamma}]_i\right) \right|\\&\le \sum_{j=1}^n |c_{i,j}|\int_0^t \left|\mathbb{E}_{\mu}([X_{s}^{\gamma}]_j-[Y_{s}^{\gamma}]_j) \right|\mathrm{d}s+\int_0^t\left|\mathbb{E}_{\mu}((b(X_{s}^{\gamma})-\widetilde{b}(X_s^{\gamma}))\cdot e_i)\right| \\
    &\le \sum_{j=1}^n |c_{i,j}|\int_0^t \underset{1\le k\le n}{\sup}\left|\mathbb{E}_{\mu}([X_{s}^{\gamma}]_k-[Y_{s}^{\gamma}]_k) \right|\mathrm{d}s& \\ &\quad \quad \quad+\underset{1\le k \le n}{\sup}\int_0^t\left|\mathbb{E}_{\mu}((b(X_{s}^{\gamma})-\widetilde{b}(X_s^{\gamma}))\cdot e_k)\right| \\
    &\le \sum_{l,j=1}^n |c_{i,j}|\int_0^t \underset{1\le k\le n}{\sup}\left|\mathbb{E}_{\mu}([X_{s}^{\gamma}]_k-[Y_{s}^{\gamma}]_k) \right|\mathrm{d}s \\ & \quad \quad \quad+\underset{1\le k \le n}{\sup}\int_0^t\left|\mathbb{E}_{\mu}((b(X_{s}^{\gamma})-\widetilde{b}(X_s^{\gamma}))\cdot e_k)\right|
\end{align*}
We notice that $i$ does not appear in the last expression, thus we have: \begin{align*}\underset{1\le k\le n}{\sup}\left|\mathbb{E}_{\mu}\left([X_{t}^{\gamma}]_k-[Y_{t}^{\gamma}]_k\right) \right|&\le\|C\|_1\int_0^t \underset{1\le k\le n}{\sup}\left|\mathbb{E}_{\mu}([X_{s}^{\gamma}]_k-[Y_{s}^{\gamma}]_k) \right|\mathrm{d}s\\ & \quad \quad+\underset{1\le k \le n}{\sup}\int_0^t\left|\mathbb{E}_{\mu}((b(X_{s}^{\gamma})-\widetilde{b}(X_s^{\gamma}))\cdot e_k)\right|.\end{align*}
We notice that for $t\in [0,T]$, we have $$\underset{1\le k \le n}{\sup}\int_0^t\left|\mathbb{E}_{\mu}((b(X_{s}^{\gamma})-\widetilde{b}(X_s^{\gamma}))\cdot e_k)\right|\le C_T^{\gamma}$$
where \begin{equation}\label{C_T}C_T^{\gamma}=\underset{1\le k \le n}{\sup}\int_0^T\left|\mathbb{E}_{\mu}\left((b(X_{s}^{\gamma})-\widetilde{b}(X_s^{\gamma}))\cdot e_k\right)\right|.\end{equation}

Finally, writing $g_{\gamma}:t\in \R_+\mapsto \underset{1\le k\le n}{\sup}\left|\mathbb{E}_{\mu}(X_{t,k}^{\gamma}-Y_{t,k})\right|\in \R_+$, we have, for $t\in[0,T]$, that: $$g_{\gamma}(t)\le \|C\|_1 \int_{0}^{t} g_{\gamma}(s) \mathrm{d}s + C_T^{\gamma},$$
so Grönwall's lemma gives that for $t\in [0,T]$: \begin{equation}\label{minoration}0\le g_{\gamma}(t)\le  C_T^{\gamma} e^{\|C\|_1 t}.\end{equation}
But Corollary 1 gives that for any $s>0$, we have that $X_s^{\gamma}$ lives around points of $K_0$ when $\gamma$ is big enough, and we have defined $\widetilde{b}$ in such a way that $\widetilde{b}_{|K_0}=b_{|K_0}$. Eventually we have for all $s>0$: $$\left|\mathbb{E}_{\mu}\left((b(X_{s}^{\gamma})-\widetilde{b}(X_s^{\gamma}))\cdot e_k\right)\right|\underset{\gamma\to+\infty}{\longrightarrow}0.$$ This quantity is bounded by a constant (that is integrable on $[0,T]$), by definition of $C_T^{\gamma}$ in \eqref{C_T} the dominated convergence theorem gives us eventually:
$$C_T^{\gamma}\underset{\gamma\to+\infty}{\longrightarrow} 0.$$ 

Therefore, by \eqref{minoration} we finally have:
$$g_{\gamma}(t)\underset{\gamma\to+\infty}{\longrightarrow} 0.$$

Thus, for $g:x\in \R^n\mapsto v\cdot x+l\in \R$, where $v=(v_i)_{1\le i\le n}\in \R^n$ and $e\in \R$, one has: \begin{align*}
    \left|\langle \mu, e^{t\La_{\gamma}}g \rangle- \langle \mu, e^{t\mathcal{G}_{\gamma}}g \rangle\right|&=\left|\mathbb{E}_{\mu}(g(X_t^{\gamma}))-\mathbb{E}_{\mu}(g(X_t^{\gamma})) \right| \\
    &=\left|\mathbb{E}(X_t^{\gamma}-Y_t^{\gamma})\cdot v\right|\\
    &=\left|\sum_{1\le i \le n} v_i \mathbb{E}([X_{t}^{\gamma}]_i-[Y_{t}^{\gamma}]_i)\right|\\
    &=\underset{1\le i \le n}{\sup}\left|\mathbb{E}([X_{t}^{\gamma}]_i-[Y_{t}^{\gamma}]_i) \right|\times \sum_{1\le j \le n} |v_j| \\
    &=g_{\gamma}(t) \|v\|_1 \\
    &\underset{\gamma\to+\infty}{\longrightarrow} 0 \numberthis \label{approxlin}
\end{align*}

Furthermore, since the result of Corollary 1 is uniform in $\mu$, this convergence result is uniform in $\mu$ for $g$ affine and $t>0$ given. \\

Eventually, we have that for any affine function $g$: \begin{align*}
    \langle \mu,e^{t\La_{\gamma}}g \rangle&\underset{\gamma\to+\infty}{\longrightarrow} \langle \overline{\mu}, e^{t\overline{\La}} \overline{g}\rangle \\
    &=\sum_{z\in K_0}\sum_{y\in K_0} \overline{\mu}(y) e^{t\overline{\La}}(y,z) \overline{g}(z).
\end{align*}

We will now use the geometrical hypothesis of the theorem to conclude the proof for a Lipschitz test function $f$. Since we have $n+1$ independent points in a $n$-dimensional vector space, there exists an affine function $g$ such that $f_{|K_0}=g_{|K_0}$. \\

We thus have $\overline{f}=\overline{g}$, so in particular we have proved that: $$\langle \mu, e^{t\La_{\gamma}} g \rangle\underset{\gamma\to+\infty}{\longrightarrow}\langle \overline{\mu}, e^{t\overline{\La}} \overline{g} \rangle=\langle \overline{\mu}, e^{t\overline{\La}} \overline{f} \rangle.$$

Therefore if we prove that: $$\left|\langle \mu, e^{t\La_{\gamma}} f \rangle-\langle \mu, e^{t\La_{\gamma}} g \rangle \right| \underset{\gamma\to+\infty}{\longrightarrow} 0,$$
we will be done with the theorem.\\

Let $\varepsilon>0$, we consider $\eta>0$ sufficiently small in order that the balls $B(z,\eta)$ for $z\in K_0$ are disjoint and such that: $$\underset{z\in K_0}{\sup}\underset{x\in B(z,\eta)}{\sup} |f(x)-g(x)|\le \varepsilon.$$

On the other hand, Corollary 1 gives that for $\gamma$ big enough $$\Py_{\mu}\left(X_t^{\gamma}\notin \underset{z\in K_0}{\bigcup} B(z,\eta) \right)\le \varepsilon.$$

Eventually, we have for $\gamma$ big enough that:
\begin{align*}
    \left|\langle \mu, e^{t\La_{\gamma}} f \rangle-\langle \mu, e^{t\La_{\gamma}} g \rangle \right|&=\left|\mathbb{E}_{\mu}((f-g)(X_t^{\gamma}))\right| \\
    &\le (\|f\|_{\infty}+\|g\|_{\infty})\Py_{\mu}\left(X_t^{\gamma}\notin \underset{z\in K_0}{\bigcup} B(z,\eta) \right)\\ &+\sum_{z\in K_0} \Py_{\mu}\left(X_t^{\gamma}\in B(z,\eta) \right)\underset{x\in B(z,\eta)}{\sup} |f(x)-g(x)|\\
    &\le (\|f\|_{\infty}+\|g\|_{\infty})\varepsilon+\varepsilon.
\end{align*}
This proves the theorem.
\end{proof}
\subsection{Proof of Theorem 3}
Now that our Theorem \ref{mtheorem} is proven, let us give a sketch of the proof of Theorem \ref{gtheorem}.\\

We know that for all $\gamma>0$ our trajectory $(X_t^{\gamma})_{t\ge 0}$ are elements of $\mathbb{D}$, so \cite{MZ84} gives us a powerful criterion to show that this family of processes is tight in $\mathbb{L}^0(\R,K)$. Thus, Prokhorov's theorem gives us that the set of the laws of theses processes is a relatively compact subset of the space of probability measures on $\mathbb{L}^0(\R,K)$ for the topology of weak convergence. \\

To prove Theorem \ref{gtheorem}, we therefore just have to show that if a subsequence of $\left((X_t^{\gamma})_{t\ge 0}\right)_{\gamma>0}$ weakly converges to $\beta\in \mathbb{L}^0$, then its limit is necessarily $(\overline{X}_t)_{t\ge 0}$. It is in fact enough to prove that for any $r\in \N^*$ and for almost any sequence $0 \le t_1\le \ldots\le t_r<+\infty$ one has that the laws of $(\overline{X}_{t_1},\ldots,\overline{X}_{t_r})$ and $(\beta_{t_1},\ldots,\beta_{t_r})$ are the same.\\

Let us start with the following result:

\begin{theorem}\label{critere}\cite[Theorem 4]{MZ84}
Let $E$ be an Euclidean space. If $X$ is an $E$-valued stochastic process with natural filtration $(\mathcal{F}_t, t\ge 0)$, then for any $\tau\in \R_+$, its conditional variation on $[0,\tau]$ is defined as: $$V_{\tau}(X):=\underset{0=t_0<t_1<\ldots<t_k=r}{\sup}\sum_{i=0}^{k-1} \mathbb{E}\left( \|\mathbb{E}(X_{t_{i+1}}-X_{t_i}|\mathcal{F}_{t_i})\|\right)$$
Consider an index set $I$ and a family $(X^{\gamma},\gamma>0)$ of processes living in $\mathbb{D}(\R_+,E)$ which satisfy:
$$\underset{\gamma>0}{\sup}\left| V_{\tau}(X^{\gamma})+\mathbb{E}\left(\underset{0\le t\le \tau}{\sup} X_t^{\gamma}\right) \right|<+\infty,$$
for all $\tau>0$. Then the family of laws of the $X^{\gamma}$ is tight for the Meyer--Zheng topology and all the limiting points are supported in $\mathbb{D}(\R_+,E)$.
\end{theorem}
The previous theorem is the main tool for proving:
\begin{lemme}
The family of processes: $$(X_t^{\gamma})_{t\ge 0}, \gamma>0$$
is tight in the Polish space $(\mathbb{L}^0(\R_+,K),d)$ and all limiting points are supported on càdlàg paths.
\end{lemme}
\begin{proof}~\\
Since for any $\gamma>0$ and any $t\ge 0$, one has $X_t^{\gamma}\in K$ almost surely where $K$ is compact, we only have to prove that the conditional variation in Meyer--Zheng's Theorem \ref{critere} is bounded on segments. For fixed $\tau>0$, we have in our case:
$$V_{\tau}(X^{\gamma}):=\underset{0=t_0<t_1<\ldots<t_k=r}{\sup}\sum_{i=0}^{k-1} \mathbb{E}\left( \|\mathbb{E}(X_{t_{i+1}}^{\gamma}-X_{t_i}^{\gamma}|\mathcal{F}_{t_i})\|\right).$$
Equivalently, thanks to \cite[Eq. (4),(5)]{MZ84} and following paragraph,
$$V_{\tau}(X^{\gamma})=\underset{\|\varphi\|_2\le C_K}{\sup}\int_0^{\tau}\mathbb{E}\left(\langle\varphi_t,\mathrm{d}X_t^{\gamma} \rangle\right)$$
where the supremum is taken over the simple predictable process taking value in the ball of center $0$ and of radius $C_K:=\underset{z\in K}{\sup}\|z\|_2$. It follows from $\eqref{EDS}$, using that the mean of a Brownian motion is equal to $0$, that: $$V_{\tau}(X^{\gamma})=\underset{\|\varphi\|\le C_K}{\sup}\int_0^{\tau}\mathbb{E}(\langle \varphi_t, b(X_t^{\gamma}) \rangle)\mathrm{d}t.$$
So Cauchy-Schwarz inequality gives us: $$V_{\tau}(X^{\gamma})\le C_K\int_{0}^{\tau}\mathbb{E}(\|b(X_t^{\gamma})\|_2)\mathrm{d}t.$$
Since $b$ is bounded by $M$, this quantity is trivially bounded by $C_K\times M\times \tau$ for all $\gamma>0$, and we have our result using Theorem \ref{critere}.
\end{proof}
Now that we proved that $(X^{\gamma})_{\gamma>0}$ is tight, we only have to show that if a subsequence converges, it converges to $\overline{X}$ to conclude. So from now on we assume that $(X^{\gamma_p})_{p\in \N}$ converges weakly to some $\beta$, where $(\gamma_p)_{p\in \N}$ is an unbounded sequence in $]0,+\infty[$. Theorem \ref{critere} gives us furthermore that the trajectories of $\beta$ are almost surely càdlàg.  \\

In order to prove that the law of $\beta$ is the law of $\overline{X}$, we will in fact show that almost all their finite-dimensional distributions are the same, that is to say that for all $r\in \N^*$, for almost all $0\le t_1\le \ldots \le t_r<+\infty$ and for all $f$ continuous function on $K^r$ one has: $$\mathbb{E}\left(f(\beta_{t_1},\ldots,\beta_{t_r} )\right)=\mathbb{E}\left(f(\overline{X}_{t_1},\ldots,\overline{X}_{t_r}) \right).$$

Let us start with a linearization trick:
\begin{lemme}\label{approx2}
For all $f$ continuous on $K^r$, there exists $F:(\R^n)^r\longrightarrow \R$ a $r$-linear function such that: 
\begin{align*}&\int_{[0,+\infty[^r}\left|\mathbb{E}\left(f(X_{t_1}^{\gamma_p},\ldots, X_{t_r}^{\gamma_p})\right)-\mathbb{E}\left(F(X_{t_1}^{\gamma_p},\ldots, X_{t_r}^{\gamma_p})\right)\right|\lambda^{\otimes r}(\mathrm{d}t_1,\ldots,\mathrm{d}t_r)\\& \underset{p\to+\infty}{\longrightarrow}0\end{align*}
where we remind that $\lambda(\mathrm{d}t)=e^{-t}\mathrm{d}t.$
\end{lemme}
\begin{proof}
Let us first introduce the function $F$. We write $z_0,\ldots,z_n$ the elements of $K_0$ and for all $\textbf{i}=(i_1,\ldots,i_r)\in\llbracket 0,n \rrbracket^r$ we write $F_{\textbf{i}}=f(z_{i_1},\ldots,z_{i_{r}})$. We notice that for all $(x_1,\ldots,x_r)\in \{z_i, \ i\in\llbracket 0,n \rrbracket\}^r$ one has:
\begin{align*}
    f(x_1,\ldots,x_r)&=\sum_{\textbf{i}\in\llbracket 0,n\rrbracket^r} F_{\textbf{i}}\prod_{k=1}^r \mathbb{1}_{z_{i_k}}(x_k) \\
    &=\sum_{\textbf{i}\in\llbracket 0,n\rrbracket^r} F_{\textbf{i}} \prod_{k=1}^r f_{z_{i_k}}(x_k) \\
    &=: F(x_1,\ldots,x_r),
\end{align*}
where the functions $f_{z_{i_k}}$ are the one mentioned in Remark \ref{example}. Since they are all linear, the function $F$ is clearly a continuous $r$-linear map. Now \cite[Theorem 6]{MZ84} applied to $f$ and $F$ says exactly that: \begin{align*}&\int_{[0,+\infty[^r}|\mathbb{E}(f(X_{t_1}^{\gamma_p},\ldots,X_{t_r}^{\gamma_p}))-\mathbb{E}(f(\beta_{t_1},\ldots,\beta_{t_r}))|\lambda^{\otimes r}(\mathrm{d}t_1,\ldots,\mathrm{d}t_r)\\&\underset{p\to+\infty}{\longrightarrow} 0 \numberthis \label{res1} \end{align*}
and:
\begin{align*}
&\int_{[0,+\infty[^r}|\mathbb{E}(F(X_{t_1}^{\gamma_p},\ldots,X_{t_r}^{\gamma_p}))-\mathbb{E}(F(\beta_{t_1},\ldots,\beta_{t_r}))|\lambda^{\otimes r}(\mathrm{d}t_1,\ldots,\mathrm{d}t_r)\\&\underset{p\to+\infty}{\longrightarrow} 0 \numberthis \label{res2}
\end{align*}
We will now invoke an immediate consequence of Corollary \ref{corollary}: for any sequence $0\le t_1 \le \ldots \le t_r<+\infty$ and for any $\varepsilon>0$ one has, using $\Py(A\cap B)\ge \Py(A)+\Py(B)-1$, that: \begin{align*}&\Py_{\mu}\left(X_{t_1}^{\gamma}\in \underset{z\in K_0}{\bigcup} B(z,\eta),\ldots,X_{t_r}^{\gamma}\in \underset{z\in K_0}{\bigcup} B(z,\varepsilon)\right)\\&\ge \Py_{\mu}\left(X_{t_1}^{\gamma}\in \underset{z\in K_0}{\bigcup} B(z,\varepsilon),\ldots,X_{t_{r-1}}^{\gamma}\in \underset{z\in K_0}{\bigcup} B(z,\varepsilon)\right)\\ & \quad+\Py_{\mu}\left(X_{t_r}^{\gamma}\in \underset{z\in K_0}{\bigcup} B(z,\varepsilon)\right)-1  \\
&\ge \left(\sum_{k=1}^r \Py_{\mu}\left(X_{t_k}^{\gamma}\in \underset{z\in K_0}{\bigcup} B(z,\varepsilon)\right)\right)-(r-1) \ \ \ \ \ \text{by recurrence} \\
&\underset{\gamma\to+\infty}{\longrightarrow} 1.
\end{align*}
Thus, since $X^{\gamma_p}$ converges weakly to $\beta$, one has that $(\beta_{t_1},\ldots,\beta_{t_r})\in K_0^r$ almost surely, so eventually one has almost surely: $$f(\beta_{t_1},\ldots,\beta_{t_r})=F(\beta_{t_1},\ldots,\beta_{t_r}).$$
Since this is true almost surely, the expectancy of these two quantities are equal, and combining \eqref{res1} and \eqref{res2} we find:
\begin{align*}
    &\int_{[0,+\infty[^r}|\mathbb{E}(f(X_{t_1}^{\gamma_p},\ldots,X_{t_r}^{\gamma_p}))-\mathbb{E}(F(X_{t_1}^{\gamma_p},\ldots,X_{t_r}^{\gamma_p}))|\lambda^{\otimes r}(\mathrm{d}t_1,\ldots,\mathrm{d}t_r)\\&\underset{p\to+\infty}{\longrightarrow}0 \numberthis \label{res3}
\end{align*}
Therefore, in the large $p$ limit, every continuous function $f$ in $r$ variables can be replaced by its $r$-linearization.
\end{proof}
The $r$-linearization of $f$ is a sum of terms of the form $\prod_{k=1}^r \limits f_{z_{i_k}}$, so we will first study these elementary bricks. One has for $F=\prod_{k=1}^r \limits f_{z_{i_k}}$ that:
\begin{align*}
    &\mathbb{E}_{\overline{\mu}}\left(F(\overline{X}_{t_1},\ldots, \overline{X}_{t_r})\right)\\&=\mathbb{E}_{\overline{\mu}}\left(\prod_{i=1}^r f_{z_{i_k}}(\overline{X}_{t_k}) \right) \\
    &=\mathbb{E}_{\overline{\mu}}\left(\mathbb{E}_{\overline{\mu}}\left(\prod_{i=1}^r f_{z_{i_k}}(\overline{X}_{t_k}) |\overline{X}_{t_{r-1}}=z_{i_{r-1}}\right)\right)\\
    &=\mathbb{E}_{\overline{\mu}}(f_{z_r}(\overline{X}_{t_r})|\overline{X}_{t_{r-1}}=z_{i_{r-1}})\times\mathbb{E}_{\overline{\mu}}\left(\prod_{i=1}^{r-1} f_{z_{i_k}}(\overline{X}_{t_k}) \right) \\
    &=\mathbb{E}_{\overline{\mu}}\left(f_{z_{i_1}}(X_{t_1})\right)\times \prod_{k=2}^r \mathbb{E}_{\mu}\left( f_{z_k}(\overline{X}_{t_k}|\overline{X}_{t_{k-1}}=z_{i_{k-1}})\right) \\
    &=\langle \overline{\mu},e^{t_1}f_{z_1} \rangle\prod_{k=2}^r \langle \delta_{z_{i_{k-1}}}, e^{(t_{k}-t_{k-1})\overline{\La}} f_{z_{i_k}}\rangle \ \ \ \ \text{by Theorem \ref{gtheorem}} \\
    &=\sum_{i=0}^n \overline{\mu}(z_i)\prod_{k=1}^r e^{(t_k-t_{k-1})\overline{\La}}f_{z_k}(z_{i_{k-1}} )
\end{align*}
where we implicitly assumed that $i_0=i$. \\

To prove that this is the limit of $\mathbb{E}\left(\prod_{k=1}^r \limits f_z(X_{t_k}^{\gamma_p})\right)$ for almost any sequence $0\le t_1\le \ldots \le t_r$, we will prove by induction that:
\begin{align*}
    &0=\underset{p\to+\infty}{\lim} \int_{0=t_0\le t_1\le \ldots \le t_r<+\infty} \lambda^{\otimes r}(\mathrm{d}t_1,\ldots,\mathrm{d}t_r) \\ & \quad \left|\mathbb{E}\left(\prod_{k=1}^r f_{z_{i_k}}(X_{t_k}^{\gamma_p}) \right)-\sum_{i=0}^n \overline{\mu}(z_i)\prod_{k=1}^r e^{(t_k-t_{k-1})\overline{\La}}f_{z_k}(z_{i_{k-1}} ) \right| \numberthis \label{moyenne}
\end{align*}
For $r=1$, this is trivial since the integrand is bounded by $2$ that is integrable on $\left((\mathbb{R}_+)^r,\lambda^{\otimes r}\right)$ and it converges to $0$ pointwisely using Theorem \ref{mtheorem}, we conclude using the dominated convergence theorem. \\

We assume that \eqref{moyenne} is proven for $r\in \N^*$ given. Let $(\mathcal{F}_t, t\ge 0)$ be the natural filtration. For all $0\le t_1 \le \ldots \le t_{r+1}<+\infty$, the tower property of conditional expectation and then the Markov property imply:
\begin{align*}
    \mathbb{E}\left(\prod_{k=1}^{r+1} f_{z_{i_{k}}}(X_{t_k}^{\gamma_p}) \right)&=\mathbb{E}\left(\prod_{k=1}^{r} f_{z_{i_{k}}}(X_{t_k}^{\gamma_p}) \times \mathbb{E}(f_{z_{i_{r+1}}}(X_{t_{r+1}}^{\gamma_p})|\mathcal{F}_{t_r} )\right) \\
    &=\mathbb{E}\left(\prod_{k=1}^{r} f_{z_{i_{k}}}(X_{t_k}^{\gamma_p}) \times \langle \delta_{X_{t_r}^{\gamma_p}},e^{(t_{r+1}-t_r)\La_{\gamma_p}}f_{z_{i_{r+1}}}\rangle\right)
\end{align*}
We have therefore:
\begin{align*}
    &\left|\mathbb{E}\left(\prod_{k=1}^{r+1} f_{z_{i_k}}(X_{t_k}^{\gamma_p})\right)-\mathbb{E}\left(\prod_{k=1}^{r} f_{z_{i_k}}(X_{t_k}^{\gamma_p}) \times \langle \delta_{z_{i_r}}, e^{(t_{r+1}-t_r)\La_{\gamma_p}}  f_{z_{i_{r+1}}}\rangle\right) \right| \\
    &=\left| \mathbb{E}\left(\prod_{k=1}^{r} f_{z_{i_k}}(X_{t_k}^{\gamma_p}) \times \langle \delta_{X_{t_r}^{\gamma_p}}, e^{(t_{r+1}-t_r)\La_{\gamma_p}} f_{z_{i_{r+1}}} \rangle\right)\right. \\ &\quad \quad \left. -\mathbb{E}\left(\prod_{k=1}^{r} f_{z_{i_k}}(X_{t_k}^{\gamma_p})  \times \langle \delta_{z_{i_r}}, e^{(t_{r+1}-t_r)\La_{\gamma_p}}f_{z_{i_{r+1}}}  \rangle \right) \right| \\
    &=\left| \mathbb{E}\left(\prod_{k=1}^{r} f_{z_{i_k}}(X_{t_k}^{\gamma_p})  \right. \right. \\ &\quad \quad \quad \left.\left. \times \left(\langle \delta_{X_{t_r}^{\gamma_p}}, e^{(t_{r+1}-t_r)\La_{\gamma_p}}f_{z_{i_{r+1}}} \rangle -\langle \delta_{z_{i_r}}, e^{(t_{r+1}-t_r)\La_{\gamma_p}}f_{z_{i_{r+1}}}  \rangle\right)\right) \right| \\
    &\le\mathbb{E}\left(\left|f_{z_{i_r}}(X_{t_r}^{\gamma_p})\right.\right. \\ &\left. \left. \quad\quad\times \left(\langle \delta_{X_{t_r}^{\gamma_p}}, e^{(t_{r+1}-t_r)\La_{\gamma_p}}f_{z_{i_{r+1}}} \rangle -\langle \delta_{z_{i_r}}, e^{(t_{r+1}-t_r)\La_{\gamma_p}}f_{z_{i_{r+1}}}  \rangle\right) \right| \right)
\numberthis \label{sfinal}\end{align*}
where the penultimate inequality follows $|f_{z_{i_k}}(X_{t_k}^{\gamma_p})|\le 1$. We want to show that \eqref{sfinal} converges to $0$. In order to do that, we will consider approximations of $\langle \delta_{X_{t_r}^{\gamma_p}}, e^{(t_{r+1}-t_r)\La_{\gamma_p}}f_{z_{i_{r+1}}} \rangle$ and $\langle \delta_{z_{i_r}}, e^{(t_{r+1}-t_r)\La_{\gamma_p}}f_{z_{i_{r+1}}}  \rangle$ where $\gamma_p$ only appears in the bras.\\

The key idea is the one used in the proof of Theorem \ref{mtheorem}: instead of considering a complex drift $b$, we approximate our process of generator $\La_{\gamma}$ with a one of generator $\mathcal{G}_{\gamma}$ whose drift $\widetilde{b}$ is linear. Since $f_{z_{i_{r+1}}}$ is affine and the convergence in \eqref{approxlin} is uniform in $\mu$, one has:
$$
    \left|\langle \delta_{X_{t_r}^{\gamma_p}}, e^{(t_{r+1}-t_r)\La_{\gamma_p}}f_{z_{i_{r+1}}} \rangle-\langle \delta_{X_{t_r}^{\gamma_p}}, e^{(t_{r+1}-t_r)\mathcal{G}_{\gamma_p}}f_{z_{i_{r+1}}} \rangle \right|\underset{p\to+\infty}{\longrightarrow} 0, 
$$
and
$$\left|\langle \delta_{z_{i_r}}, e^{(t_{r+1}-t_r)\La_{\gamma_p}}f_{z_{i_{r+1}}}\rangle-\langle \delta_{z_{i_r}}, e^{(t_{r+1}-t_r)\mathcal{G}_{\gamma_p}}f_{z_{i_{r+1}}}\rangle \right|\underset{p\to+\infty}{\longrightarrow} 0,$$
where the first convergence is uniform in $\omega$ in the sample space.\\

We can therefore replace $\La_{\gamma}$ by $\mathcal{G}_{\gamma}$ in \eqref{sfinal}, but since $\widetilde{b}$ is linear, Lemma \ref{elin} gives us: $$\langle \delta_{z_{i_r}}, e^{(t_{r+1}-t_r)\La_{\gamma_p}}f_{z_{i_{r+1}}}\rangle=\langle \delta_{z_{i_r}}, e^{(t_{r+1}-t_r)\widetilde{\La^{(0)}}}f_{z_{i_{r+1}}}\rangle$$
and
$$\langle \delta_{X_{t_r}^{\gamma_p}}, e^{(t_{r+1}-t_r)\La_{\gamma_p}}f_{z_{i_{r+1}}}\rangle=\langle \delta_{X_{t_r}^{\gamma_p}}, e^{(t_{r+1}-t_r)\widetilde{\La^{(0)}}}f_{z_{i_{r+1}}}\rangle$$
We eventually integrate \eqref{sfinal} and use the last two equalities:
\begin{align*}
    &\underset{p\to+\infty}{\limsup} \int_{0=t_0\le t_1\le \ldots \le t_{r+1}<+\infty} \lambda^{\otimes r}(\mathrm{d}t_1,\ldots,\mathrm{d}t_{r+1})\\ &\left|\mathbb{E}\left(\prod_{k=1}^{r+1} f_{z_{i_k}}(X_{t_k}^{\gamma_p})\right)-\mathbb{E}\left(\prod_{k=1}^{r} f_{z_{i_k}}(X_{t_k}^{\gamma_p}) \times \langle \delta_{z_{i_r}}, e^{(t_{r+1}-t_r)\La_{\gamma_p}}  f_{z_{i_{r+1}}}\rangle\right) \right|  \\
    &\le \underset{p\to+\infty}{\limsup} \int_{0}^{+\infty} \mathbb{E}\left(\left|f_{z_{i_r}}(X_{t_r}^{\gamma_p}) \right.\right. \\ &\quad \left.\left.\times \left(\langle \delta_{X_{t_r}^{\gamma_p}}, e^{(t_{r+1}-t_r)\widetilde{\La^{(0)}}}f_{z_{i_{r+1}}}\rangle -\langle \delta_{z_{i_r}}, e^{(t_{r+1}-t_r)\widetilde{\La^{(0)}}}f_{z_{i_{r+1}}}\rangle\right) \right| \right) \lambda(\mathrm{d}t_{r+1}) \\
    &= \int_{0}^{+\infty} \mathbb{E}\left(\left|f_{z_{i_r}}(\beta_{t_r})\right.\right. \\ & \quad \left.\left.\times \left(\langle \delta_{\beta_{t_r}}, e^{(t_{r+1}-t_r)\widetilde{\La^{(0)}}}f_{z_{i_{r+1}}}\rangle -\langle \delta_{z_{i_r}}, e^{(t_{r+1}-t_r)\widetilde{\La^{(0)}}}f_{z_{i_{r+1}}}\rangle\right) \right| \right) \\
    &=0,
\end{align*}
the last equality being a consequence of $f_{z_{i_r}}(\beta_{t_r})=\mathbb{1}_{z_{i_r}}(\beta_{t_r})$ for all $t_r$, so that the product with the other term vanished necessarily. Now invoking the induction hypothesis with $r$:
\begin{align*}
    0=&\underset{p\to+\infty}{\limsup} \int_{0=t_0\le t_1\le \ldots \le t_{r+1}<+\infty}\left|\sum_{i=0}^n \overline{\mu}(z_i)\prod_{k=1}^{r+1} e^{(t_k-t_{k-1})\overline{\La}}f_{z_k}(z_{i_{k-1}} ) \right. \\ &\left. -\mathbb{E}\left(\prod_{k=1}^{r} f_{z_{i_k}}(X_{t_k}^{\gamma_p}) \times \langle \delta_{z_{i_r}}, e^{(t_{r+1}-t_r)\La_{\gamma_p}}  f_{z_{i_{r+1}}}\rangle\right)  \right| \lambda^{\otimes r}(\mathrm{d}t_1,\ldots,\mathrm{d}t_{r+1})
\end{align*}
Combining the last limits, we prove the claim for $r+1$.\\

Using that $F$ is a sum of terms of the form $\prod_{k=1}^r \limits f_{z_{i_k}}$, equation \eqref{moyenne} implies:
\begin{align*}
    0&=\underset{p\to+\infty}{\lim} \int_{0=t_0\le t_1\le \ldots \le t_r<+\infty} \left|\mathbb{E}\left(F(X_{t_1}^{\gamma_p},\ldots, X_{t_r}^{\gamma_p}) \right)\right.\\& \ \ \ \ \   \ \ \ \ \left.-\sum_{\textbf{i}\in\llbracket 0,n \rrbracket^r} F_{\textbf{i}}\sum_{i=0}^n \overline{\mu}(z_i)\prod_{k=1}^r e^{(t_k-t_{k-1})\overline{\La}}f_{z_k}(z_{i_{k-1}} ) \right| \lambda^{\otimes r}(\mathrm{d}t_1,\ldots,\mathrm{d}t_r) \numberthis \label{cfinal}
\end{align*}
This last result combined with equations \eqref{res2} and \eqref{res3} implies that for $\lambda^{\otimes r}$-almost every (so for Lebesgue-almost every) sequence $0\le t_1\le\ldots \le t_r$ and for any function $f$ continuous, we have:\begin{align*}
    &\underset{p\to+\infty}{\lim} \mathbb{E}\left(f(X_{t_1}^{\gamma_p}, \ldots, X_{t_r}^{\gamma_p})\right) \\
    &=\underset{p\to+\infty}{\lim} \mathbb{E}\left(F(X_{t_1}^{\gamma_p}, \ldots, X_{t_r}^{\gamma_p})\right) \\
    &=\mathbb{E}\left(F(\beta_{t_1}, \ldots, \beta_{t_r})\right) \\
    &=\underset{p\to+\infty}{\lim} \mathbb{E}\left(F(X_{t_1}^{\gamma_p}, \ldots, X_{t_r}^{\gamma_p})\right) \\
    &=\sum_{\textbf{i}\in\llbracket 0,n \rrbracket^r} F_{\textbf{i}}\sum_{i=0}^n \overline{\mu}(z_i)\prod_{k=1}^r e^{(t_k-t_{k-1})\overline{\La}}f_{z_k}(z_{i_{k-1}} )
\end{align*}
Thus, almost all the finite-dimensional distributions of $\beta$ and $\overline{X}$ are the same. Finally, \cite[Theorem 6]{MZ84} yields that $\beta$ has the same law as the processes $\overline{X}$ where $\overline{X}$ is the Markov process on $K_0$ of generator $\overline{\La}$ and of initial condition $\overline{\mu}$. This proves Theorem \ref{gtheorem}.

\section{Discussion on the geometrical hypothesis}
\label{sec:CE}

The proof of our theorem massively use affine approximations, first of the drift then of our test functions, and thus falls apart if we do not assume that $K_0$ is composed of $n+1$ affinely independent points.\\

We may however ask ourselves if it is only a technical hypothesis, or if it is completely essential, that is to say Theorem \ref{mtheorem} is false when it is not fulfilled. In this section, we actually construct a counterexample of our theorem with a compact set $K$ of $\R^2$ and a diffusion process on it where the cardinal of $K_0$ is equal to $4>2+1$.\\

Let us give a graphic representation of our counterexample:
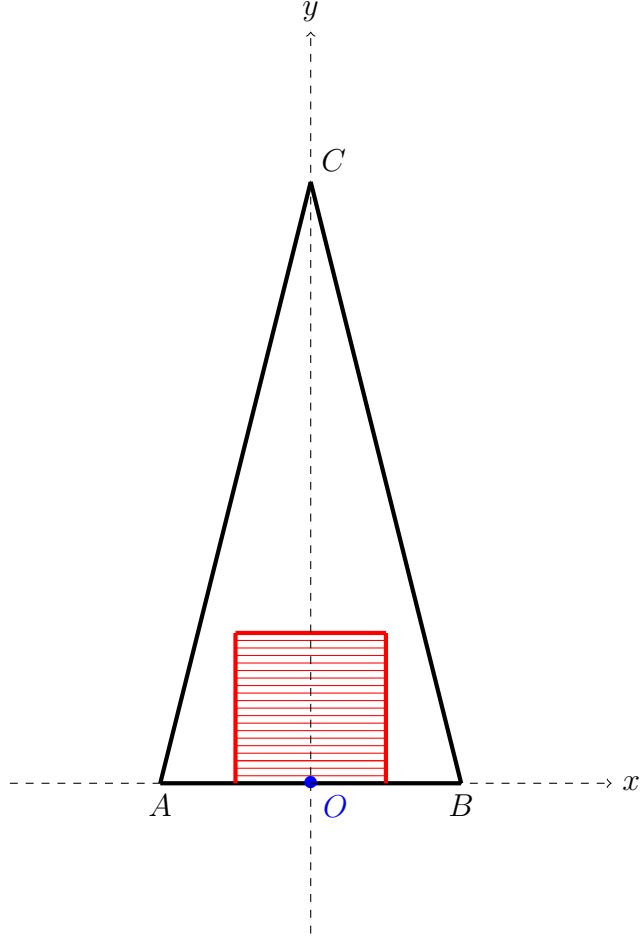
\begin{center}
\begin{tikzpicture}[scale=2]
\draw[ultra thick] (0,0) -- (2,0);       
\draw[ultra thick] (0,0) -- (1,4);   
\draw[ultra thick] (2,0) -- (1,4);  
\draw (0,0) node [below] {$A$};
\draw (2,0) node [below] {$B$};
\draw (1,4) node [above right] {$C$};

\draw[red, ultra thick] (0.5,1) -- (1.5,1);
\draw[red] (0.5,0.95) -- (1.5,0.95);
\draw[red] (0.5,0.90) -- (1.5,0.90);
\draw[red] (0.5,0.85) -- (1.5,0.85);
\draw[red] (0.5,0.80) -- (1.5,0.80);
\draw[red] (0.5,0.75) -- (1.5,0.75);
\draw[red] (0.5,0.70) -- (1.5,0.70);
\draw[red] (0.5,0.65) -- (1.5,0.65);
\draw[red] (0.5,0.60) -- (1.5,0.60);
\draw[red] (0.5,0.55) -- (1.5,0.55);
\draw[red] (0.5,0.50) -- (1.5,0.50);
\draw[red] (0.5,0.45) -- (1.5,0.45);
\draw[red] (0.5,0.40) -- (1.5,0.40);
\draw[red] (0.5,0.35) -- (1.5,0.35);
\draw[red] (0.5,0.30) -- (1.5,0.30);
\draw[red] (0.5,0.25) -- (1.5,0.25);
\draw[red] (0.5,0.20) -- (1.5,0.20);
\draw[red] (0.5,0.15) -- (1.5,0.15);
\draw[red] (0.5,0.15) -- (1.5,0.15);
\draw[red] (0.5,0.1) -- (1.5,0.1);
\draw[red] (0.5,0.05) -- (1.5,0.05);
\draw[blue, ultra thick] (1,0) node[below right] {$O$} node{$\bullet$};
\draw[->, dashed] (-1,0)--(3,0) node[right]{$x$};
\draw[->, dashed] (1,-1)--(1,5) node[above]{$y$};
\draw[ultra thick, red] (0.5,0)--(0.5,1);
\draw[ultra thick, red] (1.5,0)--(1.5,1);
\end{tikzpicture}
\end{center}
Here, our compact $K$ is the solid triangle of vertices $A(-1,0)$, $B(1,0)$, and $C(0,4)$. Let us define the dominant noise coefficient $\sigma$. For $(x,y)\in [-\frac{1}{2},\frac{1}{2}]\times]0,1]$, we take: 
$$\sigma(x,y)=(x^2(1-x)(1+x)+y^2)\begin{pmatrix}
1 & 0 \\
0 & 0
\end{pmatrix}$$
If we furthermore take for $x\in [-1,1]$ : $$\sigma(x,0)=(x^2(1-x)(1+x)+y^2)\begin{pmatrix}
1 & 0 \\
0 & 0
\end{pmatrix}$$
we have defined a smooth function on $[-1,1]\times \{0\} \cup ]-\frac{1}{2},\frac{1}{2}]\times]0,1]$ that is null in $A,B$ and $O$. While it is not very interesting to write it down explicitly, it is really not hard to find a smooth extension of $\sigma$ that will be null only on $A,B,C$ and $O$. We have thus $K_0=\{A,B,C,O\}$.\\

To understand where the problem is, it may be interesting to consider the functions $H_z$ for $z\in K_0$. The key hypothesis of the (uniform) ergodic theorem was that the functions $H_z$ were continuous on $K$. We claim this is false in our example.\\

Let us consider $H_{O}$: on the one hand, one has obviously that $H_{O}(O)=1$. On the other hand, let us consider the dominant diffusion process $(X_t)_{t\ge 0}$ solution of \eqref{eq:purenoise-007}, starting from $(0,y)$ with $y>0$. Then, for $\omega$ an element of the sample space, if $X_t(\omega)$ converges to $O$ when $t$ approaches infinity, for any $\varepsilon>0$, one has $X_t(\omega)\in B(O,\varepsilon)$ for all $t$ big enough. But in fact, the red lines of our drawing are insurmountable obstacles: since the noise on the $y$-axis is null on it, if $X_t(\omega)$ is on one of them at a given time $t$, in order to decrease along the $y$-axis it has to go out of the red box and enter in it again. Yet we know that $X_t(\omega)\in B(O,\frac{1}{2})$ for $t$ big enough, and thus $X_t(\omega)$ cannot get out of the box $[-\frac{1}{2},\frac{1}{2}]\times[0,1]$. Therefore, we have $$X_t(\omega)\in \left[-\frac{1}{2},\frac{1}{2}\right]\times\{0\}$$ for $t$ big enough, that is to say we reached the bottom line in finite time. But since we have to quit the box to get to the line, there exists $T$ such that $$X_{T}(\omega)\in \left(\left[-1,\frac{1}{2}\right]\cup \left[\frac{1}{2},1\right]\right)\times \{0\}.$$ Starting from any point of this set, the probability of converging to $O$ is less or equal to $\frac{1}{2}$, so using Markov strong property (the process in homogeneous in time), we have that the probability of converging to $O$ starting from $(0,y)$ is less or equal to $\frac{1}{2}$ for any $y>0$: our function $H_O$ cannot be continuous.

\begin{remark}\label{r5}
If the noise is Lipschitz, it is in fact impossible to reach the border in finite time, so we have actually $H_z(x,y)=0$ for all $y>0$: the state $O$ is unreachable except if one starts from the bottom line.
\end{remark}

We easily see that Theorem \ref{mtheorem} cannot hold in this situation, since our objects are not even well defined: to compute matrix $\overline{\La}$, one has to consider the quantity: $$b(O)\cdot \nabla H_O(O),$$
that is not well defined if $b$ is non null on the $y$-axis.\\

Furthermore, even if our process nevertheless was converging to a stochastic process, there are situations where we know that the latter could not be a continuous-time Markov process with state space $K_0$. We consider the noise $\sigma$ defined above (that is only null on $A,B,C$ and $O$), and a smooth drift $b$ that checks: $$\left\{\begin{array}{ll}
     b(x,y)=\begin{pmatrix}
 0 \\
 1
\end{pmatrix}  & \text{if $(x,y)\in [-\frac{1}{2},\frac{1}{2}]\times[0,1]$}, \\
     b\in \mathbb{R}_+ \begin{pmatrix} 0 \\
1
\end{pmatrix}, \\
b=0 &\text{outside the box $[-\frac{5}{9},\frac{5}{9}]\times[0,\frac{10}{9}]$}.
\end{array}\right.$$ 
It is easy to prove the existence of such a drift. We assume that the stochastic process $(X_t^{\gamma})_{t\ge 0}$ starting from $O$ of generator $\La_{\gamma}$ converges in the weak sense of Theorem \ref{mtheorem} to the Markov process $(\widetilde{X}_t)_{t\ge 0}$ starting from $O$ of state space $K_0$ and of generator $\overline{\La}$, and we will get a contradiction.\\

The idea of the proof is the following: since $(\widetilde{X}_s)_{s\ge 0}$ is a Markov process starting from $O$, we have $\widetilde{X}_t\in O$ with a probability approaching $1$ for a time $t>0$ small enough. Let us now consider the process $(X_s^{\gamma})_{s\ge 0}$ when $\gamma$ is big enough. During $[0,t]$ it first enters the box $[-\frac{1}{2},\frac{1}{2}]\times]0,1]$ pushed by the drift in $O$, then, since the noise inside is very strong, it quits the red box and, from the exit point, finally approaches $A$, $B$ or $C$ with a probability greater than a strictly positive quantity. Since this one is independent of $\gamma$ and $t$ when $\gamma$ is big enough, it will contradict the fact that it should converge to $0$ when $\gamma$ approaches infinity and $t$ approaches $0$.\\

We consider a smooth function $f$ such that $f(O)=1$. We assume furthermore that $0\le f\le 1$ and that $f$ is null outside $B(O,\tfrac{1}{8})$. We know that there exists $t>0$ small enough such that: $$\langle \delta_O, e^{t\overline{\La}} \overline{f} \rangle \ge 1-2^{-10}.$$ 
Hence for $\gamma$ big enough: \begin{equation}\label{A}\mathbb{P}\left(X_t^{\gamma}\in B(O,\tfrac{1}{8})\right)\ge\langle \delta_O, e^{t\La_{\gamma}} f\rangle \ge 1-2^{-9}\end{equation}

Let us now consider the process $(X_s^{\gamma})_{s\ge 0}$ starting from $O$ of generator $\La_{\gamma}$. We write $T^{\gamma}$ the hitting time of the boundary of our red square: $$\inf\left\{s>0, X_s^{\gamma}\in \left\{-\frac{1}{2}\right\}\times[0,1]\bigcup\left\{\frac{1}{2}\right\}\times[0,1]\bigcup \left[-\frac{1}{2},\frac{1}{2}\right]\times \{1\}\right\}.$$

We may now prove that it is arbitrarily small with a high probability when $\gamma$ approaches infinity. Indeed let $g_{\varepsilon}$ be a smooth function such that $g_{\varepsilon}(A)=g_{\varepsilon}(B)=g_{\varepsilon}(C)=g_{\varepsilon}(O)=1$ and we furthermore assumed that $0\le g_{\varepsilon}\le 1$ and that $g_{\varepsilon}$ is null outside of the balls of center $z\in K_0$ and of radius $\varepsilon$. Then for any $s>0$ and for any $\varepsilon$ one has that \begin{equation}\label{B}\mathbb{P}\left(X_s^{\gamma}\in \underset{z\in K_0}{\bigcup}B(z,\varepsilon)\right)\ge\langle \delta_O, e^{s\La_{\gamma}}g_{\varepsilon} \rangle\underset{\gamma\to+\infty}{\longrightarrow} \langle \delta_O, e^{s\overline{\La}}\overline{g_{\varepsilon}} \rangle=1 .\end{equation}
Now let us assume that there exists an increasing sequence $(\gamma_k)_k$ going to infinity such that $$\mathbb{P}(T_{\gamma_k}>s)\ge \alpha>0.$$
We notice that for $\omega$ in the sample space such that $T_{\gamma}(\omega)>s$, using that $\sigma\cdot e_y=0$ into the red box, we have that: $$X_s^{\gamma}(\omega)\cdot e_y=\int_{0}^s b(X_u^{\gamma}(\omega))\mathrm{d}u\cdot e_y= s.$$ 
Since $X_s^{\gamma}(\omega)$ is in the red box, we have thus: $$
    \mathbb{P}\left(X_s^{\gamma_k}\notin \underset{z\in K_0}{\bigcup}B(z,\frac{s}{2})\right)\ge \mathbb{P}(T_{\gamma_k}>s)\ge\alpha$$
for all $k\in \N$, but it contradicts $\eqref{B}$ when $k$ approaches infinity.\\

In conclusion, for all $s>0$, one has for $\gamma$ large enough that $\mathbb{P}(T_{\gamma}\le s)\ge \frac{1}{2}$, and we notice that for $\omega$ such that $T_{\gamma}(\omega)\le s$, one has \begin{equation}\label{C}
    X_{T_{\gamma}(\omega)}^{\gamma}(\omega)\cdot e_y=\int_{0}^{T_{\gamma}(\omega)} b(X_u^{\gamma}(\omega))\cdot e_y\mathrm{d}u=T_{\gamma}(\omega)\le s.
\end{equation}
We now consider the two points $R_{-1}(-\frac{1}{2},0)$ and $R_1(\frac{1}{2},0)$. We just proved that choosing $\gamma$ big enough, we can make sure that $(X_u^{\gamma})_{u\ge 0}$ reaches a point on the border of the red box arbitrarily close to $R_{-1}$ or $R_1$ during $[0,t]$ with a probability at least $\frac{1}{2}$. We will thus make two approximations: we will first approach during a short time the  diffusion process starting from the point $x_{\gamma}=X_{T_{\gamma}}^{\gamma}$ with the dominant one starting from the same point (using Lemma \ref{approx1}), and then approach the dominant process starting from $x_{\gamma}$ with the one starting from $R_{-1}$ or $R_1$.\\

Since both cases are symmetric, from now on we assume that $x_{\gamma}\in \{\frac{1}{2}\}\times[0,1]$. We consider $(X_s (z))_{s\ge 0}$ the process of generator $\La^{(1)}$ starting from $z\in K$. We have with a probability $\frac{1}{2}$ that $X_s (R_1)\underset{s\to+\infty}{\longrightarrow}B$, so there exists $\beta>0$ such that: $$\mathbb{P}\left(X_{\beta}(R_1)\in B(B,\tfrac{1}{32})\left|\right.X_{\infty}=B\right)\ge \frac{3}{4} , $$  
and therefore:\begin{align*}&\mathbb{P}\left(X_{\beta}(R_1)\in B(B,\tfrac{1}{32})\right)\\&\ge\mathbb{P}\left(X_{\beta}(R_1)\in B(B,\tfrac{1}{32})\left|\right.X_{\infty}=B\right)\mathbb{P}(X_{\infty}(R_1)=B)\\&\ge \frac{3}{4}\times \frac{1}{2}\\&\ge\frac{3}{8} . \numberthis \label{D}\end{align*}
We now use the dependence on initial conditions: there exists $\eta\in ]0,\frac{t}{2}[$ such that for all $x\in B(R_1,\eta)$, one has: $$\underset{u\in [0,\beta]}{\sup} \mathbb{E}(\left\|X_u(R_1)-X_u(x) \right\|_2)\le \frac{1}{128}.$$
so by Markov's inequality, for all $x\in B(R_1,\eta)$: \begin{equation}\label{E}\mathbb{P}\left(\|X_{\beta}(R_1)-X_{\beta}(x)\|_2\ge \frac{1}{32} \right)\le \frac{1}{8}.\end{equation}
We now use $\eqref{B}$: for $\gamma$ large enough, taking $s=\eta$, we have, with a probability at least $\frac{1}{2}$, that $T_{\gamma}\le \eta$ and thus $X_{T_{\gamma}}^{\gamma}\in B(R_1,\eta)$ (respectively $B(R_{-1},\eta)$).\\

We now write $\gamma=\frac{\beta}{h}$, and we take $\gamma$ big enough so that, using Lemma \ref{approx1}, we have for all $x\in B(R_1,\eta)$: $$\mathbb{E}\left(\|X_{h}^{\frac{\beta}{h}}(x)-X_{\beta}(x)\|_2\right)\le \frac{1}{128}$$ and therefore: \begin{equation}\label{F}\mathbb{P}\left(\|X_{h}^{\frac{\beta}{h}}(x)-X_{\beta}(x)\|_2\ge\frac{1}{16}\right)\le \frac{1}{8}.\end{equation}
Finally we have using $\eqref{D}$, $\eqref{E}$ and $\eqref{F}$: \begin{align*}
    &\mathbb{P}\left(X_{h}^{\gamma}(x)\notin B(B,\tfrac{1}{8})\right)\\&\le \mathbb{P}\left(\{X_\beta(R_1)\notin B(B,\frac{1}{32})\}\bigcup\{\|X_{\beta}(R_1)-X_{\beta}(x)\|_2\ge \frac{1}{32}|\}\right.\\ & \left. \quad \quad \quad \quad\bigcup\{\|X_{\beta}(x)-X_{h}^{\frac{\gamma}{h}}(x)\|_2\ge \frac{1}{16}|\}\right) \\
    &\le \frac{5}{8}+\frac{1}{8}+\frac{1}{8} \\
    &\le \frac{7}{8}.
\end{align*}
Thus we have: \begin{align*}
    &\mathbb{P}\left(\{X_{T_{\gamma}+h}\in B(A,\tfrac{1}{8})\cup B(B,\tfrac{1}{8})\}\bigcap \{T_{\gamma}\le \frac{t}{2}\}\right)\\ &\ge \mathbb{P}\left(\{X_{T_{\gamma}+h}\in B(A,\tfrac{1}{8})\cup B(B,\tfrac{1}{8})\}\bigcap \{T_{\gamma}\le \eta\}\right) \\
    &\ge \mathbb{P}\left(X_{T_{\gamma}+h}\in B(A,\tfrac{1}{8})\cup B(B,\tfrac{1}{8})| T_{\gamma}\le \eta\right)\mathbb{P}\left(T_{\gamma}\le \eta \right) \\
    &\ge \frac{1}{8}\times \frac{1}{2}\\
    &\ge \frac{1}{16} . \numberthis \label{G}
\end{align*}
We thus have with a probability at least $\frac{1}{16}$ that at a time $T_{\gamma}+h<t$ our process will be in balls of radius $\tfrac{1}{8}$ of centers $A$ or $B$, this result being true for all $\gamma$ big enough. We just have to prove that during $[T_{\gamma}+h,t]$, our process stays, with a fix strictly positive probability, far from the ball $B(O,\tfrac{1}{8})$ to find a contradiction with $\eqref{A}$. The idea of the proof is to show that if one is close to $B$, the probability that it stays close to $B$ forever is strictly positive. \\

We consider $S$ the stopping defined by: $$S(x)=\inf\left\{s\ge0, X_s(x)\notin B(B,\tfrac{1}{4})\right\}.$$
We notice that the definition $S$ involves the dominant process and not the general one, and for a good reason: as long as our process lives outside the box $[-\frac{5}{9},\frac{5}{9}]\times[0,\frac{10}{9}]$, the drift is null and our complex process acts exactly like the dominant one. We have for all $x\in B(B,\tfrac{1}{8})$ that $(X_{S(x)\wedge u}(x))_{u\ge 0}$ is a martingale, so we have that: $$x=\mathbb{E}(X_0(x))=\mathbb{E}(X_{S(x)}(x)),$$ with $S(x)$ being potentially infinite. Since $(X_{u}(x))_{u\ge 0}$ converges (it is a bounded martingale), if $S(x)(\omega)=+\infty$ then $X_{\infty}(x)(\omega)=B$. We have therefore, rewriting the last equation, that: \begin{equation}\label{H}x=\mathbb{P}(S(x)=+\infty)B+\mathbb{P}(S(x)<+\infty)\mathbb{E}(X_{S(x)}(x)|S(x)<+\infty).\end{equation}
Using basic geometry we have that $W(x)=\mathbb{E}(X_{S(x)}(x)|S(x)<+\infty)$, being the barycenter of points from $\partial B(B,\frac{1}{4})\cap K$, checks \begin{equation}\label{I}\|W(x)-B\|_2\ge \frac{\sqrt{2}}{8}.\end{equation}
Since $x\in B(B,\tfrac{1}{8})$, one has using $\eqref{H}$ and $\eqref{I}$ that: \begin{equation}\label{J}\mathbb{P}(S(x)=+\infty)\ge \frac{\sqrt{2}-1}{\sqrt{2}}>\frac{1}{4}\end{equation}
Our dominant process starting from $x\in B(B,\tfrac{1}{8})$ does not quit $B(B,\frac{1}{4})$ with a probability at least $\frac{1}{4}$. Yet, if it never quits $B(B,\frac{1}{4})$, it will never reach $B(O,\tfrac{1}{8})$, so using $\eqref{G}$ and $\eqref{J}$ we get for $\gamma$ big enough: \begin{align*}
    &\mathbb{P}\left(X_t^{\gamma}\notin B(O,\tfrac{1}{8})\right)\\&\ge \mathbb{P}\left(\{X_t^{\gamma}\notin B(O,\tfrac{1}{8})\}\bigcap \{X_{T_{\gamma}+h}^{\gamma}\in B(A,\tfrac{1}{8})\cup B(B,\tfrac{1}{8})\}\bigcap\{T_{\gamma}\le \frac{t}{2}\}\right) \\&\ge \mathbb{P}\left(X_t^{\gamma}\notin B(O,\tfrac{1}{8})| X_{T_{\gamma}+h}^{\gamma}\in B(A,\tfrac{1}{8})\cup B(B,\tfrac{1}{8}),T_{\gamma}\le \frac{t}{2}\right)\\ & \ \ \ \ \ \ \ \ \mathbb{P}\left(\{X_{T_{\gamma}+h}^{\gamma}\in B(A,\tfrac{1}{8})\cup B(B,\tfrac{1}{8})\}\bigcap\{T_{\gamma}\le \frac{t}{2}\}\right) \\
    &\ge \mathbb{P}\left(S(X_{T_{\gamma}+h}^{\gamma})=+\infty|X_{T_{\gamma}+h}^{\gamma}\in B(A,\tfrac{1}{8})\cup B(B,\tfrac{1}{8}),T_{\gamma}\le \frac{t}{2}\right)\\ &\ \ \ \ \ \ \ \ \mathbb{P}\left(\{X_{T_{\gamma}+h}^{\gamma}\in B(A,\tfrac{1}{8})\cup B(B,\tfrac{1}{8})\}\bigcap\{T_{\gamma}\le \frac{t}{2}\}\right) \\
    &\ge \frac{1}{4}\times \frac{1}{16} \\
    &\ge 2^{-6} \numberthis \label{K}
\end{align*} 
Finally $\eqref{A}$ gives us $\mathbb{P}\left(X_t^{\gamma}\notin B(O,\tfrac{1}{8})\right)\le 2^{-9}$ and \eqref{K} gives us $\mathbb{P}\left(X_t^{\gamma}\notin B(O,\tfrac{1}{8})\right)\ge 2^{-6}$: Contradiction.\\

We get that one cannot write the potential limit of our stochastic process of generator $\La_{\gamma}$, as $\gamma$ approaches infinity, as a pure jump Markov process of state space $K_0$.

\begin{remark}\label{r6}
In this counterexample, we notice that $O$ belongs to the segment $[A,B]$. More generally, it is easy to build a counterexample of the theorem when the $(n+2)^{\text{th}}$ element of $K_0$ belongs to the convex envelope of the other points: if our extra point is in the interior of $K$, we just have to squeeze it between two red boxes like the one of our counterexample to find similar contradictions. We have however not yet found a counterexample of our theorem when there is at least $n+2$ points but we assumed that the the points of $K_0$ are exactly the extremal points of $K$.
\end{remark}
\begin{remark}\label{r7}
One may say that our counterexample is a pathological one since our point $O$ is unreachable if one does not start from a point of the boundary, that is say unreachable from almost any starting point. We may indeed imagine that our limit process could be, from almost any starting point, a Markov chain on the points $\{A,B,C\}$. However, and it is quite surprising, a point that almost surely does not exist for the dominant process, may in fact exists for the complex process. Indeed, if one takes on the compact $K$ of our counterexample: $$
     b(A)=\begin{pmatrix}
1  \\
0 
\end{pmatrix} , \
     b(B)=\begin{pmatrix}
-1  \\
0 
\end{pmatrix} , \
b(C)=\begin{pmatrix}
 0 \\
 -1\end{pmatrix}, \ b(O)=0,$$
and the same $\sigma$ as above, then we observe in the limit process an increasing accumulation of the mass on $O$, whatever the starting point is. We cannot thus make as if $O$ was not existing.
\end{remark}

\section*{Appendix: Proof of the Ergodic Theorem (Theorem \ref{ergodique})}
\label{sec:Appendix}

We denote by ${\mathcal P} (K)$ the set of probability measures on $K$. We recall that $K$ is equipped with the usual Euclidean norm $\Vert \cdot \Vert_2$. For $\mu,\nu \in {\mathcal P} (K)$, the (first) Wasserstein's distance between $\mu$ and $\nu$ is defined as:
$$d_W(\mu,\nu)=\sup \left\{\int_{K} \varphi(z)(\mathrm{d}\mu(z)-\mathrm{d}\nu(z))\left| \   \text{$\varphi$ Lipschitz},\ \textbf{Lip}(\varphi)\le 1 \right. \right\}.$$
We recall that $({\mathcal P} (K), d_W)$ is Polish space and that $d_W$ metrizes the weak convergence on $\mathcal P (K)$.\\

From now on we write $H(x,\cdot) \in {\mathcal P} (K)$ the law of $X_{\infty}(x)$. We have then by definition of the functions  $H_z$ that for any Lipschitz  function $f$ on $K$: $$\Pa f(x)=\sum_{z\in K_0} H_z(x) f(z)=\int_K f(z)H(x,\mathrm{d}z).$$

If we assume that the functions $H_z$ are continuous, we have easily that the function $x\in K\mapsto H(x,\cdot) \in {\mathcal P} (K)$ is continuous.

\begin{remark}
In this article we only consider $K_0$ finite, but our proof of the uniform ergodic theorem would still work whatever is the form of the null set of $\sigma$ as long as the dominant process is pure noise (that is to say $(X_t)_{t\ge 0}$ is a martingale) and the function $x\mapsto H(x,\cdot)$ is continuous (where $H(x,\cdot)$ is still the law of $X_{\infty}(x)$).
\end{remark}
Let us start with a technical lemma:

\begin{lemme}
\label{continuite}
If $x\in K\mapsto H(x,\cdot) \in {\mathcal P} (K)$ is continuous, then for every $\alpha,h>0$, there exists $\beta>0$ such that for every $z,y\in K$, $$d_W(\delta_z,H(y,\cdot))\le \beta \ \ \ \ \text{implies } \ \ \  \mathbb{P}(X_{\infty}(y)\in B(z,h))\ge 1-\alpha.$$
\end{lemme}

\begin{proof} ~\\
We consider a smooth function $\zeta_z$ taking values in $[0,1]$ such that  $\zeta_z (z)=0$ and the restriction of $\zeta_z$ to the complement of $B(z,h)$ is $1$. Being smooth, it is $L$-Lipschitz, therefore $\frac{1}{L} \zeta_z$ is $1$-Lipschitz. Eventually since $\zeta_z(z)=0$: \begin{align*}
    \mathbb{P}(X_{\infty}(y)\notin B(z,h))&\le \mathbb{E}\left(\zeta(X_{\infty}(y))\right) \\
    &\le L \left(\mathbb{E}\left(\frac{1}{L}\zeta_z(X_{\infty}(y))\right)-\mathbb{E}\left(\frac{1}{L}\zeta_z(z)\right)\right) \\ 
    &\le L d_W (H(y,\cdot),\delta_z).
\end{align*}
We therefore only have to chose $\beta\le \frac{h}{L}$ to conclude. The constant $L$ is independent of $z$, since we could perform the same computation for $z'\in K$ considering $\zeta_{z'}=\zeta_z(\cdot-(z'-z))$ that is also $L$-Lipschitz.
\end{proof}
We may now prove the uniform ergodic theorem:
\begin{proof}(Theorem \ref{ergodique})~\\
We will first prove the simple convergence of our process.\\

Since $\La^{(1)}$ is just noise, one has for all $x$ and for all $i\in \llbracket 1,n \rrbracket$ that $$X_{t,i}(x)=X_{t}(x) \cdot e_i$$ is a martingale that lives in the compact space $K$ and is therefore bounded. \\

The martingale convergence theorem implies that for any $x$, $(X_{t,i}(x))_{t\ge 0}$ converges almost surely and in $L^2$ to $[X_{\infty}(x)]_i$ as $t$ goes to infinity. Let us prove that $(X_{t,i}(x))_{t\ge 0}$ converges necessarily to one of the $z\in K_0$. From now on the $x$ will be implicit. \\

Using Dambis, Dubins-Schwarz's theorem, there exists an extension $\widetilde{\Omega}$ of our probability space $\Omega$ and a Brownian motion on this space $\beta$ such that: $$X_{t,i}={\beta_{\langle X_{i} \rangle}}_t$$ where the quadratic variation of $X_{i}$ is given by: 
\begin{align*}\langle X_{i} \rangle_t
=\sum_{j=1}^n\int_{0}^t \sigma_{i,j}\left(X_s\right)^2 \mathrm{d}s .
\end{align*}

Thus, since for almost any $\omega\in \widetilde{\Omega}$ one has that $X_{t,i}(\omega)$ converges when $t$ tends to infinity, we have that $\langle X_{i} \rangle_t(\omega)$ converges. \\

But $\langle X_{i} \rangle_t$ is a sum of $n$ integrals from $0$ to $t$ of positive functions, that is to say $n$ functions increasing with $t$ and thus for all $i,j\in\llbracket 1,n\rrbracket$ one has: $$t\mapsto\int_{0}^t \sigma_{i,j}\left(X_s(\omega)\right)^2 \mathrm{d}s$$ converges and thus $$\sigma_{i,j}\left(X_t(\omega)\right)\underset{t\to+\infty}{\longrightarrow} 0.$$

Eventually one has that: $$\sigma\left(X_{\infty}(\omega)\right)=0,$$ 
and thus $X_{\infty}\left(\omega\right)\in K_0$.\\

Therefore we get that for all $x\in K$: \begin{align*}
    \mathbb{E}_x\left(f(X_{\infty})\right)&=\int_{K} f(z)H(x,\mathrm{d}z)            \\&=\sum_{z\in K_0} \mathbb{P}_x(X_{\infty}=z)f(z) \\
    &=\mathcal{P}f(x).
\end{align*}
Finally since $f$ is $L$-Lipschitz for some $L\ge 0$:
 \begin{align*}
    \left|\mathbb{E}_x\left(f(X_t)\right)-\mathcal{P}f(x)\right|&\le L\mathbb{E}\left(\left\|X_t(x)-X_{\infty}(x)\right\|_2\right) \\
    &\le L\mathbb{E}\left(\left\|X_t(x)-X_{\infty}(x)\right\|_2^2\right)^{\frac{1}{2}} \\
    &\underset{t\to+\infty}{\longrightarrow}0,
\end{align*}
by using the theorem of convergence of martingales in $L^2$.\\

We thus proved the pointwise convergence but not the uniform one.\\

We notice that for any $t>0$, $x,x'\in K$, one has: \begin{align*}
    &\left|e^{t\La^{(1)}}f(x')-\Pa f(x') \right|\\&\le L\mathbb{E}\left(\|X_t(x)-X_t(x')\|_2\right)+L\mathbb{E}\left(\|X_t(x)-X_{\infty}(x)\|_2\right)\\& \ \ \ \ \ \ \ \ +L \mathbb{E}\left(\|X_{\infty}(x)-X_{\infty}(x')\|_2\right) \\
    &\le L \mathbb{E}\left(\|X_t(x)-X_t(x')\|_2^2\right)^{\frac{1}{2}}+L \mathbb{E}\left(\|X_t(x)-X_{\infty}(x)\|_2^2\right)^{\frac{1}{2}}\\& \ \ \ \ \ \ \ \ +L \mathbb{E}\left(\|X_{\infty}(x)-X_{\infty}(x')\|_2^2\right)^{\frac{1}{2}} \\
    &\le L \mathbb{E}\left(\|X_t(x)-X_{\infty}(x)\|_2^2\right)^{\frac{1}{2}}+2L \mathbb{E}\left(\|X_{\infty}(x)-X_{\infty}(x')\|^2_2\right)^{\frac{1}{2}}
\end{align*}
where we used in the last inequality that $\|X_t(x)-X_{t}(x')\|_2^2$ is a submartingale, and thus its expectation is increasing with $t$.\\

Let $\varepsilon>0$, using the last inequality, we have that if there is $\eta_x>0$ such that for all $x'\in B(x,\eta_x)$ one has: \begin{equation}\label{eqn1}\mathbb{E}\left(\|X_{\infty}(x')-X_{\infty}(x)\|_2^2\right)\le \varepsilon^2,\end{equation}
then for $t>0$ such that $$\mathbb{E}\left(\|X_t(x)-X_{\infty}(x)\|_2^2\right)^{\frac{1}{2}}\le \varepsilon,$$ we will have for all $x'\in B(x,\eta_x)$: $$ \left|e^{t\La^{(1)}}f(x')-\Pa f(x') \right|\le 3L\varepsilon.$$

Since $K$ is a compact set, it can be covered by a finite number of balls of the form $B(x,\eta_x)$, yet we have for all $x$ that: $$\mathbb{E}\left(\|X_t(x)-X_{\infty}(x)\|_2^2\right)\underset{t\to+\infty}{\longrightarrow}0,$$ so eventually there exists $M>0$ such that for all $t>M$ and for all $x'\in K$ we have: $$ \left|e^{t\La^{(1)}}f(x')-\Pa f(x') \right|\le 3L\varepsilon,$$
and the theorem is proved.\\

We will now prove $\eqref{eqn1}$. Let $\alpha,h$ be in  $]0,1[$, Lemma \ref{continuite} gives us that there exists $\beta>0$, such that \begin{equation*}d_W(\delta_z,H(y,\cdot))\le \beta \ \ \ \ \text{implies } \ \ \  \mathbb{P}(X_{\infty}(y)\in B(z,h))\ge 1-\alpha.\end{equation*}  The function $x \in K \mapsto H(x,\cdot) \in {\mathcal P} (K) $ is continuous  and since $K$ is compact, Heine's theorem gives that it is uniformly continuous. Therefore there exists $a>0$ such that $|x-y|\le a$ implies $d_W (H(x,\cdot),H(y,\cdot))\le \beta$. Since for $z\in K_0$ it holds that $H(z,\cdot)=\delta_z$, we get that for all $z\in K_0$: \begin{equation}
    |z-y|\le a \ \ \ \ \ \text{implies} \ \ \ \ \ \ \mathbb{P}(X_{\infty}(y)\in B(z,h))\ge 1-\alpha \label{approx}
\end{equation}

We have that $(X_t(x))_t$ converges almost surely to $X_{\infty}(x)$, therefore there exists $M>0$ such that $\mathbb{P}(X_M(x)\in B(X_{\infty}(x),\frac{a}{2}))\ge 1-\alpha$. We now use the dependence on initial conditions: there exists $\eta_x>0$ such that for all $x'\in B(x,\eta_x)$, we have: $$\mathbb{P}\left(|X_M(x')-X_M(x)|\ge\frac{a}{2}\right)\le\alpha.$$
Therefore: \begin{equation}\mathbb{P}\left(X_M(x')\in B(X_{\infty}(x),a)\right)\ge 1-2\alpha. \label{approx3}\end{equation}

By homogeneity of our process, we have: \begin{align*}&\mathbb{P}(X_{\infty}(x')\in B(X_{\infty}(x),h))\\&=\mathbb{P}(X_{\infty}(X_M(x'))\in B(X_{\infty}(x),h)) \\
&\ge\mathbb{P}(X_{\infty}(X_M(x'))\in B(X_{\infty}(x),h)|X_M(x')\in B(X_{\infty}(x),a))\\& \ \ \ \ \ \ \ \ \times\mathbb{P}(X_M(x')\in B(X_{\infty}(x),a)) \\
&\ge (1-\alpha)(1-2\alpha)  \ \ \ \ \ \text{by \eqref{approx} and \eqref{approx3}} \\
&\ge 1-3\alpha-2\alpha^2 \\
&\ge 1-5\alpha.
\end{align*} \\
Finally, if $R=\underset{z\in K}{\sup} \|z\|_2^2$, we have for all $x'\in B(x,\eta_x)$:\begin{align*}
    &\mathbb{E}(\|X_{\infty}(x)-X_{\infty}(x')\|_2^2)\\&\le 2R\  \mathbb{P}(X_{\infty}(x')\notin B(X_{\infty}(x),h))+h^2\  \mathbb{P}\left(X_{\infty}(x')\in B(X_{\infty}(x),h)\right)  \\
    &\le 2R \times 5\alpha+h^2 .
\end{align*}
We then choose $h$ and $\alpha$ small enough so that the last term is inferior to $\varepsilon^2$, which proves \eqref{eqn1} and thus our theorem.

\end{proof}

\section*{Acknowledgements}
I would like to give a huge thank to my two master thesis supervisors, Cédric Bernardin and Raphaël Chetrite, who introduced me to this problem. They were a very listening ear and supported me throughout the redaction of this article. A special thanks to C\'edric for his many proofreadings  and corrections of preliminary versions of this paper and to Raphaël for the abundant bibliography he gave me.

\end{document}